\documentclass[12pt,leqno]{amsart}
\usepackage{amsmath,amsfonts,amssymb,amscd,amsthm,amsbsy,upref}
\usepackage{color}

\usepackage{comment}
\usepackage{mathrsfs}
\usepackage{bm}

\usepackage[normalem]{ulem}

\usepackage{mdframed}

\usepackage{varioref}
\usepackage[colorlinks,allcolors=DarkBlue]{hyperref}
\usepackage[nameinlink]{cleveref}%must be loaded after all other packages

\newtheorem{thm}{Theorem}[section]
\newtheorem*{thm*}{Theorem}
\newtheorem{lem}[thm]{Lemma}
\newtheorem{prop}[thm]{Proposition}
\newtheorem{cor}[thm]{Corollary}
\newtheorem{example}[thm]{Example}
\newtheorem{notation}[thm]{Notation}
\newtheorem{prob}[thm]{Problem}
\theoremstyle{definition}
\newtheorem{defin}[thm]{Definition}
\newtheorem{remark}[thm]{Remark}
\newtheorem{convention}[thm]{Convention}

\newtheorem{question}{Question}

\numberwithin{equation}{section}

\newcommand{\N}{\mathbb N}
\newcommand{\R}{\mathbb R}

\newcommand{\E}{\mathbb E}
\renewcommand{\P}{\mathbb P}

\newcommand{\xb}{\overline{x}}

\newcommand{\vp}{\varepsilon}

\newcommand{\supp}{{\rm{supp}}}
\newcommand{\dist}{{\rm{dist}}}
\newcommand{\cof}{{\rm{cof}}}

\newcommand{\vmo}{{\rm{VMO}}}
\newcommand{\bmo}{{\rm{BMO}}}
\newcommand{\umd}{{\rm{UMD}}}

\newcommand{\ie}{{\it i.e.}}

\newcommand{\keq}{\!=\!}
\newcommand{\kleq}{\!\leq\!}
\newcommand{\kge}{\!\ge\!}
\newcommand{\kle}{\!<\!}

\newcommand{\kin}{\!\in\!}

\DeclareMathOperator{\spa}{span}
\usepackage[svgnames,x11names]{xcolor}

\usepackage{varioref}
\usepackage[nameinlink]{cleveref}%must be loaded after all other packages

\usepackage[section]{placeins}
\usepackage{float}

\usepackage{xargs}

\usepackage[colorinlistoftodos,prependcaption,textsize=tiny]{todonotes}

\newcommandx{\todoin}[2][1=]{%
  \todo[inline, caption={todo}, #1]{%
    \begin{minipage}{\textwidth-20pt}#2\end{minipage}}}

\newcommandx{\minor}[2][1=]{\todo[linecolor=olive,backgroundcolor=olive!25,bordercolor=olive,#1]{#2}}
\newcommandx{\minorin}[2][1=]{\minor[inline, caption={minor issues}, #1]{%
    \begin{minipage}{\textwidth-40pt}{\normalsize #2}\end{minipage}}}

\newcommandx{\major}[2][1=]{\todo[linecolor=blue,backgroundcolor=blue!25,bordercolor=blue,#1]{#2}}
\newcommandx{\majorin}[2][1=]{\major[inline, caption={major tasks}, #1]{%
    \begin{minipage}{\textwidth-40pt}{\normalsize #2}\end{minipage}}}

\usepackage[all]{xy}
\usepackage{enumerate}

\usepackage{xifthen}%provides ifthenelse command

\usepackage{ifthenx}%complements xifthen. MUST be loaded after xifthen

\usepackage{xargs}

\newcounter{proof}
{\stepcounter{proof}\begin{proof}}%
  {\end{proof}}%

\newcounter{proofstep}[proof]
\newenvironment{proofstep}[1][]%
{\refstepcounter{proofstep}\smallskip\par\noindent%
  \ifthenelse{\isempty{#1}}%if
  {\textsc{Step \theproofstep.}}%then
  {\textsc{#1.}}%else
  \noindent}%
{\par}%
\newcounter{proofcase}[proof]
{\refstepcounter{proofcase}\smallskip\par\noindent%
  \ifthenelse{\isempty{#1}}%if
  {\textsc{Case \theproofcase.}}%then
  {\textsc{#1.}}%else
  \noindent}%
{\par}%

\DeclareMathOperator{\cond}{\mathbb{E}}

\allowdisplaybreaks

\title{The factorization property of $\ell^\infty(X_k)$}
\author{R. Lechner}

\address{R. Lechner, Institute of Analysis, Johannes Kepler University Linz, Altenberger Strasse 69,
  A-4040 Linz, Austria}

\email{Richard.Lechner@jku.at}

\author{P. Motakis}

\address{P. Motakis, Department of Mathematics, University of Illinois at Urbana-Champaign, Urbana,
  IL 61801, USA}

\email{pmotakis@illinois.edu}

\author{P.F.X. M\"uller}

\address{P.F.X. M\"uller, Institute of Analysis, Johannes Kepler University Linz, Altenberger
  Strasse 69, A-4040 Linz, Austria}

\email{Paul.Mueller@jku.at}

\author{Th.~Schlumprecht}

\address{Th.~Schlumprecht, Department of Mathematics, Texas A\&M University, College Station, TX
  77843-3368, USA, and Faculty of Electrical Engineering, Czech Technical University in Prague,
  Zikova 4, 16627, Prague, Czech Republic}

\email{schlump@math.tamu.edu}

\thanks{The first and the third author were supported by the Austrian Science Foundation (FWF) under
  Grant Number Pr.Nr. P28352.  The second named author was supported by the National Science
  Foundation under Grant Number DMS-1912897.  The fourth named author was supported by the National
  Science Foundation under Grant Numbers DMS-1464713 and DMS-1711076.  The first and the third
  author were also supported by the 2019 workshop in Analysis and Probability at Texas A\&M
  University.}

\date{\today}

\begin{document}
\maketitle%

\begin{abstract}
  In this paper we consider the following problem: Let $X_k$, be a Banach space with a normalized
  basis $(e_{(k,j)})_j$, whose biorthogonals are denoted by $(e_{(k,j)}^*)_j$, for $k\in\N$, let
  $Z=\ell^\infty(X_k:k\kin\N)$ be their $\ell^\infty$-sum, and let $T:Z\to Z$ be a bounded linear
  operator, with a large diagonal, \ie,
  $$\inf_{k,j} \big|e^*_{(k,j)}(T(e_{(k,j)})\big|>0.$$
  Under which condition does the identity on $Z$ factor through $T$? The purpose of this paper is to
  formulate general conditions for which the answer is positive.
\end{abstract}

\tableofcontents

% redefining a command that amsart redefined
\makeatletter%
\providecommand\@dotsep{5}%
\def\listtodoname{List of Todos}%
\def\listoftodos{\@starttoc{tdo}\listtodoname}%
\makeatother%
% make list of todos

\section{Introduction}\label{S:0}

Throughout this paper, we assume $X_k$ is for each $k\in\N$ a Banach space which has a normalized
basis $(e_{(k,j)})_j$ and let $(e^*_{(k,j)})_j\subset X^*_k$ be the coordinate functionals.  Let $Z$
be the space
\begin{equation}\label{eq:intro:Z-def}
  Z
  =\ell^\infty(X_k:k\in\N)
  =\big\{(x_k) : x_k\in X_k,\,k \kin\N,\, \|(x_k)\|
  =\sup_{k\in\N} \|x_k\|_{X_k} <\infty\big\}.
\end{equation}
We say a bounded linear operator $T\colon Z\to Z$ has \emph{large diagonal}, if
\begin{equation*}
  \inf_{k,j}\big|e^*_{(k,j)}\big(Te_{(k,j)}\big)\big|
  > 0.
\end{equation*}

The main focus of this work is the following problem concerning operators on $Z$.
\begin{prob}\label{prob:1}
  Does the identity operator $I_Z$ on $Z$ factor through every bounded linear operator
  $T\colon Z\to Z$ with a large diagonal, \ie, do there exist bounded linear operators
  $A,B\colon Z\to Z$ such that $I_Z = ATB$?
\end{prob}
If \Cref{prob:1} has a positive answer, we say that $Z$ has the \emph{factorization property} (with
respect to the array $(e_{(k,j)})$).

In our previous work~\cite[Theorem~7.6]{lechner:motakis:mueller:schlumprecht:2018} we solved the
factorization problem for unconditional sums of Banach spaces with bases (e.g.\@ $\ell^p$ or $c_0$
sums).  In that case, an appropriate linear ordering of the array $(e_{(k,j)})$ is a basis of the
unconditional sum.  Since $Z$ is a non-separable Banach space, the array $(e_{(k,j)})$ \emph{cannot}
be reordered into a basis of $Z$.  In particular, we lose the norm convergence of the series
expansion of vectors in $Z$ which are not in in the $c_0$-sum of the $X_k$.  Consequently, the
arguments given in~\cite{lechner:motakis:mueller:schlumprecht:2018} are not applicable to the space
$Z$.

Historically, the first factorization problem of that type appeared in the 1967
paper~\cite{lindenstrauss:1967} by Lindenstrauss, in which he proved that the space $\ell^\infty$ is
prime.  Later in~1982, Capon~\cite{capon:1982:1} actually showed that whenever $X$ has a symmetric
basis, $\ell^\infty(X)$ has the factorization property.  Bourgain proved in his~1983
work~\cite{bourgain:1983} that $H^\infty$ is primary, by solving a factorization problem of
$\ell^\infty$-sums of finite dimensional spaces (Bourgain's localization method).  The first
applications of Bourgain's localization method appear shortly thereafter in works by the third named
author ~\cite{mueller:1988} and by Blower~\cite{blower:1990}.  The cases $X_k=L^p$,
$1 < p < \infty$, $k\in\mathbb{N}$ and $X_k = H^1$, $k\in\mathbb{N}$, were treated by
Wark~\cite{wark:2007:direct-sum} in 2007 and the third named author~\cite{mueller:2012} in 2012,
respectively.  The $\ell^\infty$-sum of mixed-norm Hardy and $\bmo$ spaces and the $\ell^\infty$-sum
of non-separable Banach spaces with a subsymmetric weak$^*$ Schauder bases were recently treated by
the first named author in~\cite{lechner:2018:subsymm,lechner:2018:factor-mixed}.

In our previous paper~\cite{lechner:motakis:mueller:schlumprecht:2018}, we developed an approach to
factorization problems based on two player games; the type of games we are referring to were first
considered by Maurey, Milman, Tomczak-Jaegermann in~\cite{maurey:milman:tomczak-jaegermann:1995} and
further developed by Odell-Schlumprecht~\cite{odell:schlumprecht:2002} and
Rosendal~\cite{rosendal:2009} who coined the term \emph{infinite asymptotic games} (see
also~\cite{odell:schlumprecht:2006,odell:schlumprecht:zsak:2007,haydon:odell:schlumprecht:2011}).
Thereby, we were able to unify the proofs of several known factorization results as well as provide
new ones.  We exploited those infinite asymptotic games to define the concept of \emph{strategically
  reproducible bases} in Banach spaces.

In the present paper, we develop a two player game approach to solve the factorization \Cref{prob:1}
on $Z$ if the array $(e_{(k,j)})_{k,j}$ is \emph{uniformly asymptotically curved}; that is, if for
every $k\in\mathbb{N}$ and every block basis $(x_{(k,j)})_j$ of $(e_{(k,j)})_j$ we have
\begin{equation}\label{eq:uac}
  \lim_n \sup_{k}\Big\|\frac{1}{n}\sum_{j=1}^n x_{(k,j)} \Big\|_{X_k} = 0.
\end{equation}
Our first main \Cref{T:2.7} isolates conditions on the array $(e_{(k,j)})$ which guarantee that
\Cref{prob:1} has a positive solution.  Moreover, if we drop the restriction that the array
$(e_{(k,j)})_{k,j}$ is uniformly asymptotically curved, then we were able to successfully treat the
following
\begin{prob}\label{prob:2}
  Does for every $T\colon Z\to Z$ with large diagonal with respect to $(e_{(k,j)})$ exist an
  infinite $\Gamma\subset\mathbb{N}$ such that the identity on
  $Z_\Gamma:=\ell^\infty( X_k : k\in\Gamma )$ factor through $T$.
\end{prob}
In the special case that $X_k = X$, $k\in\mathbb{N}$, our solution to \Cref{prob:2} implies a
positive solution to \Cref{prob:1}.

\section{Preliminaries}
\label{sec:preliminaries}

In this section, we introduce the necessary notation and concepts.

\subsection{Review of strategically reproducible bases}
\label{sec:revi-strat-repr}

Let $X$ denote a Banach space and $S\subset X$.  We define $[S]$ as the norm-closure of $\spa{S}$,
where $\spa{S}$ denotes the linear span of $S$.  Given sequences $(x_i)$ in $X$ and
$(\widetilde x_i)$ in possibly another Banach space $\widetilde X$.  We say that $(x_i)$ and
$(\widetilde x_i)$ are {\em impartially $C$-equivalent} if for any finite choice of scalars
$(a_i)\in c_{00}$ we have
\begin{equation*}
  \frac{1}{\sqrt C}\Big\|\sum_{i=1}^\infty a_i\widetilde x_i\Big\|
  \leq \Big\|\sum_{i=1}^\infty a_ix_i\Big\|
  \leq \sqrt{C}\Big\|\sum_{i=1}^\infty a_i\widetilde x_i\Big\|.
\end{equation*}
For a Banach space $X$ we denote by $\cof(X)$ the set of cofinite dimensional subspaces of $X$,
while $\cof_{w^*}(X^*)$ denotes the set of cofinite dimensionl $w^*$-closed subspaces of $X^*$.

Let $C > 0$.  Given an operator $T\colon X\to X$, we say that {\em the identity $C$-factors through
  $T$ } if there are bounded linear operators $A,B:X\to X$ with $\|A\|\|B\|\le C$ and $I = ATB$;
moreover, we say that {\em the identity almost $C$-factors through $T$ } if it
$(C+\varepsilon)$-factors through $T$ for all $\varepsilon>0$.  If an operator $T$ on $X$ satisfies
$\inf_j\big|e_j^*(Te_j)\big|>0$, then we say that $T$ has \emph{large diagonal}.  An operator $T$ on
$X$ satisfying $e^*_m(Te_j )= 0$ whenever $k\neq m$, is called a \emph{diagonal operator}.

We recall some definitions from~\cite{lechner:motakis:mueller:schlumprecht:2018}.
\begin{defin}\label{factorization definitions:basis}
  Let $X$ be a Banach space with a normalized Schauder basis $(e_j)$.
  \begin{enumerate}[(i)]
  \item\label{enu:factorization definitions:basis:i} We say that the basis $(e_j)$ has the {\em
      factorization property} if whenever $T\colon X\to X$ is a bounded linear operator with
    $\inf_j|e_j^*(Te_j)|>0$ then the identity of $X$ factors through $T$.  More precisely, for a map
    $K\colon (0,\infty)\to \R^+$ we say that $(e_j)$ has the {\em $K(\cdot)$-factorization property
      in $X$} if for every $\delta>0$ and bounded linear operator $T\colon X\to X$, with
    $\inf_j|e_j^*(Te_j)|\geq\delta$ the identity $I_X$ on $X$ almost $K(\delta)$-factors through
    $T$, \ie, for every $\vp>0$ there are operators $A,B\colon X\to X$, with
    $\|A\|\|B\|\le K(\delta)+\vp$ and $I_X=BTA$.
  \item\label{enu:factorization definitions:basis:ii} We say that the basis $(e_j)$ has the {\em
      uniform diagonal factorization property in $X$} if for every $\delta>0$ there exists
    $K(\delta)> 0$ so that for every bounded diagonal operator $T\colon X\to X$ with
    $\inf_j \big|e_j^*(Te_j)\big| \geq \delta$ the identity almost $K(\delta)$-factors through $T$.
    If we wish to be more specific we shall say that $(e_j)$ has the {\em $K(\delta)$-diagonal
      factorization property}.
  \end{enumerate}
\end{defin}

\begin{remark}\label{rem:fact-prop}
  First, we remark that if $(e_j)$ is unconditional, then it satisfies \Cref{factorization
    definitions:basis}~\eqref{enu:factorization definitions:basis:ii}.  Secondly, recall that
  by~\cite[Remark~3.11]{lechner:motakis:mueller:schlumprecht:2018} we have
  $1/\delta\leq K(\delta)\leq K(1)/\delta$.
\end{remark}

Also recall the following definition of the {\em strategic
  reproducibility}~\cite{lechner:motakis:mueller:schlumprecht:2018} of a Banach space $X$ with a
basis $(e_j)$.

\begin{defin}\label{D:2.1}
  Let $X$ be a Banach space with a normalized Schauder basis $(e_j)$ and fix positive constants
  $C\geq 1$, and $\eta>0$.

  Consider the following two-player game between Player I and Player II:

  \begin{description}
  \item[Pregame] Before the first turn Player I is allowed to choose at the beginning of the game a
    partition of $\N = N_1\cup N_2$.
  \item[Turn $n$, Step\! 1] Player I chooses $\eta_n>0$, $W_n\in\mathrm{cof}(X)$, and
    $G_n\in\mathrm{cof}_{w^*}(X^*)$,
  \item[Turn $n$, Step\! 2] Player II chooses $i_n\in\{1,2\}$, a finite subset $E_n$ of $N_{i_n}$
    and sequences of non-negative real numbers $(\lambda_i^{(n)})_{i\in E_n}$,
    $(\mu_i^{(n)})_{i\in E_n}$ satisfying
    \begin{equation*}
      1-\eta <\sum_{i\in E_n}\lambda_i^{(n)}\mu_i^{(n)}< 1+\eta.
    \end{equation*}
  \item[Turn $n$, Step\! 3] Player I chooses $(\vp_i^{(n)})_{i\in E_n}$ in $\{-1,1\}^{E_n}$.
  \end{description}
  We say that Player II has a winning strategy in the game $\mathrm{Rep}_{(X,(e_i))}(C,\eta)$ if he
  can force the following properties on the result:

  For all $j\in\N$ we set
  \begin{equation*}
    x_j
    = \sum_{i\in E_j}\vp_i^{(n)} \lambda^{(n)}_ie_i \text{ and }x_j^*
    = \sum_{i\in E_j}\vp_i^{(n)}\mu^{(n)}_i e^*_i
  \end{equation*}
  and demand:
  \begin{enumerate}[(i)]
  \item\label{enu:strat-rep:i} the sequences $(x_j)$ and $(e_j)$ are impartially
    $(C+\eta)$-equivalent,
  \item\label{enu:strat-rep:ii} the sequences $(x_j^*)$ and $(e_j^*)$ are impartially
    $(C+\eta)$-equivalent,
  \item\label{enu:strat-rep:iii} for all $j\in\N$ we have $\mathrm{dist}(x_j, W_j) < \eta_j$, and
  \item\label{enu:strat-rep:iv} for all $j\in\N$ we have $\mathrm{dist}(x^*_j, G_j) < \eta_j$.
  \end{enumerate}
  We say that $(e_j)$ is {\em $C$-strategically reproducible in $X$} if for every $\eta >0$ Player
  II has a winning strategy in the game $\mathrm{Rep}_{(X,(e_j))}(C,\eta)$.
\end{defin}
It was shown in~\cite[Remark~3.5]{lechner:motakis:mueller:schlumprecht:2018} that in the case that
$(e_j)$ is shrinking and unconditional, then \Cref{D:2.1} is equivalent to a considerably simpler
formulation.

\Cref{D:2.1} was used in \cite{lechner:motakis:mueller:schlumprecht:2018} to prove the following
factorization result:
\begin{thm}[{\cite[Theorem~3.12]{lechner:motakis:mueller:schlumprecht:2018}}]\label{T:2.3}
  Let $X$ be a Banach space with a Schauder basis $(e_j)$ that has a basis constant $\lambda$.
  Assume also that
  \begin{enumerate}[(i)]
  \item the basis $(e_j)$ has the $K(\delta)$-diagonal factorization property and
  \item the basis $(e_j)$ is $C$-strategically reproducible in $X$.
  \end{enumerate}
  Then $(e_j)$ has the $ \lambda C^{2}K(\delta)$-factorization property.
\end{thm}

\subsection{Dyadic Hardy spaces and $\bmo$}
\label{sec:review-mixed-norm}

We now turn to defining the dyadic Hardy spaces, $\bmo$ and $\vmo$.

For a more in depth discussion of the biparameter Hardy spaces, we refer
to~\cite{laustsen:lechner:mueller:2015}; see also~\cite{lechner:motakis:mueller:schlumprecht:2018}.
Let $\mathcal{D}$ denote the collection of dyadic intervals given by
\begin{equation*}
  \mathcal{D}
  = \{[k2^{-n},(k+1)2^{-n})\, :\, n,k\in \mathbb N_0, 0\leq k \leq 2^n-1\}.
\end{equation*}
For $I\in \mathcal{D}$ we let $|I|$ denote the length of the dyadic interval $I$.  Let~$h_I$ be the
$L^\infty$-normalized Haar function supported on $I\in\mathcal{D}$; that is, for $I = [a,b)$ and
$c = (a+b)/2$, we have $h_I(x) = 1$ if $a\le x <c$, $h_I(x) = -1$ if $c\leq x < b$, and $h_I(x) = 0$
otherwise.  For $1\leq p < \infty$, the \emph{Hardy space} $H^p$ is the completion of
\begin{equation*}
  \spa\{ h_{I}\, :\, I\in \mathcal{D} \}
\end{equation*}
under the square function norm
\begin{equation}\label{Hpnorm}
  \Bigl\|\sum_{I\in\mathcal{D}} a_{I} h_{I}\Bigr\|_{H^p}
  = \bigg(
  \int_0^1 \Big(
  \sum_{I\in\mathcal{D}} a_{I}^2 h_{I}^2(x)
  \Big)^{p/2}
  d x
  \bigg)^{1/p}.
\end{equation}
The Haar system $(h_{I})_{I\in\mathcal{D}}$ is a $1$-unconditional basis of~$H^p$, and thus gives
rise to a canonical lattice structure.  Finally, we define $\vmo$ as the norm closure of
$(h_I)_{I\in\mathcal{D}}$ in $(H^1)^*$, where we canonically identify $h_I$ with the linear
functional $f\mapsto \int h_I(x) f(x) dx$.

Next, let $X$ denote any Banach space.  We will now define the vector-valued Banach spaces $H^p[X]$,
$1\leq p < \infty$, $\bmo[X]$ and $\vmo[X]$.  Put
$\mathcal{D}_n = \{I\in\mathcal{D} : |I|=2^{-n}\}$,
$\mathcal{D}^n = \{I\in\mathcal{D} : |I|\geq 2^{-n}\}$, $n\geq 0$ and let $(r_I)$ denote a sequence
of independent Rademacher functions.  We define
\begin{equation}\label{eq:Hp-vector}
  H^p[X]
  = \Bigl\{ f\in L^1(X) : \|f\|_{H^p[X]} < \infty\Bigr\},
\end{equation}
where for every $f = \sum_{I\in\mathcal{D}} f_I h_I\in L^1(X)$, $f_I\in X$, $I\in\mathcal{D}$, the
norm is given by
\begin{equation*}
  \|f\|_{H^p[X]}
  = \int_0^1 \Bigl\|\sum_{I\in\mathcal{D}} r_I(t)f_I h_I\Bigr\|_{L^p(X)} dt.
\end{equation*}
Of special interest for us is the case $p=1$.  M\"uller and Schechtman observed that Davis'
inequality holds for Banach spaces with the $\umd$
property~\cite[Theorem~6]{mueller:schechtman:1991}, \ie, there exists a constant $C>0$ depending
only on the $\umd$-constant of the Banach space $X$ such that
\begin{equation*}
  C^{-1} \|f\|_{H^1[X]}
  \leq \int_0^1 \sup_{n} \|\cond_n(f)\|_X dt
  \leq C \|f\|_{H^1[X]},
\end{equation*}
where $\cond_n$ denotes the conditional expectation with respect to $\mathcal{D}_n$.  For a detailed
presentation of $\umd$ spaces, we refer to Pisier's recent monograph~\cite{pisier:2016:martingales}.

We now define $\bmo[X]$:
\begin{equation*}
  \bmo[X]
  = \Bigl\{ f\in L^1(X) : \|f\|_{\bmo[X]} < \infty\Bigr\},
\end{equation*}
where for every $f = \sum_{I\in\mathcal{D}} f_I h_I\in L^1(X)$, $f_I\in X$, $I\in\mathcal{D}$, the
norm is given by
\begin{equation*}
  \|f\|_{\bmo[X]}^2
  = \sup_{I\in\mathcal{D}} \frac{1}{|I|} \int_I \Bigl\| \sum_{J\subset I} f_J h_J(x)\Bigr\|_X^2 dx.
\end{equation*}
Taking Davis' inequality into account, we observe that for $\umd$ spaces $X$
Bourgain~\cite[Theorem~12]{bourgain:1983:vector-valued} proved that the dual of $H^1[X]$ is
$\bmo[X^*]$.  Finally, we define $\vmo[X]$ as the norm closure of
$\spa\{x_I h_I : x_I\in X,\ I\in\mathcal{D}\}$ in $\bmo[X]$.

Moreover, let $\mathcal{R} = \{I\times J: I,J\in\mathcal{D}\}$ be the collection of dyadic
rectangles contained in the unit square, and set
\begin{equation*}
  h_{I, J}(x,y) = h_I(x)h_J(y),\qquad I\times J\in\mathcal{R},\, x,y\in[0,1).
\end{equation*}

For $1\leq p,q < \infty$, the \emph{mixed-norm Hardy space} $H^p(H^q)$ is the completion of
\begin{equation*}
  \spa\{ h_{I, J}\, :\, I\times J \in \mathcal{R} \}
\end{equation*}
under the square function norm
\begin{equation}\label{HpHqnorm}
  \|f\|_{H^p(H^q)}
  = \bigg(
  \int_0^1 \Big(
  \int_0^1 \big(
  \sum_{I, J} a_{I, J}^2 h_{I, J}^2(x,y)
  \big)^{q/2}
  d y
  \Big)^{p/q}
  d x
  \bigg)^{1/p}
  ,
\end{equation}
where $f = \sum_{I, J} a_{I, J} h_{I, J}$.  The system $(h_{I, J})_{I\times J\in\mathcal{R}}$ is a
$1$-unconditional basis of~$H^p(H^q)$, called the \emph{bi-parameter Haar system}.  Note that in
view of the Khinchin-Kahane inequality, the norms in the spaces $H^p(H^q)$ and $H^p[H^q]$ are
equivalent for $1\leq p,q < \infty$; to be precise, the identity operator
$J\colon H^p(H^q)\to H^p[H^q]$ satisfies
\begin{equation*}
  \|J\|\cdot \|J^{-1}\|\leq C(p,q).
\end{equation*}
We refer to~\cite[Theorem~4, p.20]{kahane:1985}.

First note that $H^p$, $1 < p < \infty$ is a $\umd$ space; the $\umd$ constant depends only on $p$.
We will now recall that the dual of $\vmo(H^p)$, $1 < p < \infty$ is $H^1(H^{p'})$, where
$p' = \frac{p}{p-1}$.
\begin{thm}\label{thm:dual-vmo(Hp)}
  Let $1 < p < \infty$, define $p' = \frac{p}{p-1}$ and $J\colon H^1(H^{p'})\to (\vmo(H^p))^*$ by
  $f\mapsto (g\mapsto \langle f, g\rangle)$.  Then $J$ is an isomorphism with
  $\|J\|\cdot \|J^{-1}\|\leq C(p)$; hence, $\bigl(\vmo(H^p)\bigr)^* = H^1(H^{p'})$.
\end{thm}

\begin{proof}
  Define $J\colon H^1(H^{p'})\to (\vmo(H^p))^*$ by $f\mapsto (g\mapsto \langle f, g\rangle)$.
  First, we observe that by Bourgain's vector-valued version of Fefferman's
  inequality~\cite{bourgain:1983:vector-valued}, we know that $\|J\|\leq C(p)$.  Secondly, let
  $f\in H^1(H^{p'})$ be given as the finite linear combination $f = \sum_{I\in\mathcal{D}} f_I h_I$,
  $f_I\in H^{p'}$.  Define the family of functions $f_t = \sum_{I\in\mathcal{D}} r_I(t) f_I h_I$,
  where the $(r_I)$ are independent Rademacher functions.  Since
  \begin{equation*}
    \int_0^1 \|f_t\|_{L^1(H^{p'})} dt
    \geq c(p) \|f\|_{H^1(H^{p'})},
  \end{equation*}
  we find a $t_0\in [0,1]$ such that
  \begin{equation*}
    \|f_{t_0}\|_{L^1(H^{p'})}
    \geq c(p) \|f\|_{H^1(H^{p'})}.
  \end{equation*}
  Next, we choose $g\in L^\infty(H^p)$ with $\|g\|_{L^\infty(H^p)} = 1$ such that
  $\langle f_{t_0}, g\rangle = \|f_{t_0}\|_{L^1(H^{p'})}$.  Since, $f$ is a finite linear
  combination of $h_I$'s, so is $g$, and we write $g = \sum_{I\in\mathcal{D}} g_I h_I$.  Now, we
  define $h = \sum_{I\in\mathcal{D}} r_I(t_0) g_I h_I$ and note
  \begin{equation*}
    \langle f, h\rangle
    = \langle f_{t_0}, g\rangle
    = \|f_{t_0}\|_{L^1(H^{p'})}
    \geq c(p) \|f\|_{H^1(H^{p'})}.
  \end{equation*}
  Taking into account that $H^p$, $1 < p < \infty$ is a $\umd$ space, we observe
  \begin{equation*}
    \|h\|_{\bmo(H^p)}
    \leq C(p) \|g\|_{\bmo(H^p)}
    \leq 4 C(p) \|g\|_{L^\infty(H^p)}
    = 4 C(p).
  \end{equation*}
  We summarize the above calculation
  \begin{equation*}
    \|J f\|_{(\vmo(H^p))^*}
    \geq c(p) \|f\|_{H^1(H^{p'})},
    \qquad f\in H^1(H^{p'}).
  \end{equation*}

  Finally, let $L\colon \vmo(H^p)\to \mathbb{R}$ denote any bounded linear functional.  Define the
  conditional expectation $\cond_n$ by
  \begin{equation*}
    \cond_n(\sum_{I,J\in\mathcal{D}} a_{I,J} h_{I,J})
    = \sum_{I,J\in\mathcal{D}^{n-1}} a_{I,J} h_{I,J}
  \end{equation*}
  and note that $\cond_n$ is a contraction on $\bmo(H^p)$.  Next, we now define
  $h = \sum_{I,J\in\mathcal{D}} \frac{L(h_{I, J})}{|I\times J|} h_{I, J}$ and calculate
  \begin{align*}
    \|h\|_{H^1(H^{p'})}
    &= \sup_n \|\cond_{n}(h)\|_{H^1(H^{p'})}
      \leq C(p) \sup_n \sup \bigl\{\langle \cond_{n}(h), g\rangle : \|g\|_{\bmo(H^p)}\leq 1\bigr\}\\
    &= C(p)\sup_n \sup \Bigl\{
      \sum_{I,J\in\mathcal{D}^n} g_{I, J} L(h_{I, J})
      : \|g\|_{\bmo(H^p)}\leq 1
      \Bigr\}\\
    &= C(p)\sup_n \sup \Bigl\{ L\bigl( \cond_{n}(g) \bigr) : \|g\|_{\bmo(H^p)}\leq 1 \Bigr\}
      \leq C(p)\|L\|.
  \end{align*}
  It follows that $L(f) = (J h)(f)$, $f\in \vmo(H^p)$.
\end{proof}

\begin{remark}\label{remark:James}
  Later, in \Cref{lem:asymtotically-curved:1}, we will show that $H^1(H^{p'})$ does not contain
  $c_0$.  This observation allows us to give another proof of \Cref{thm:dual-vmo(Hp)}, which we will
  discuss below.  Since $H^1(H^{p'})$ has a 1-unconditional basis and it does not contain $c_0$, we
  obtain by James’ characterization~\cite[Lemma~1]{james:1950} (see
  also~\cite[Theorem~1.c.10]{lindenstrauss:tzafriri:1977}) that the biparameter Haar basis of
  $H^1(H^{p'})$ is boundedly complete.  By 1-unconditionality $H^1(H^{p'})$ is isometrically
  isomorphic to the dual of the $(H^1(H^{p'}))^*$-norm-closed linear span of
  $\{h_{I,J} : I,J\in\mathcal{D}\}$ in
  $(H^1(H^{p'}))^*$~\cite[Theorem~1.b.4]{lindenstrauss:tzafriri:1977}.  Hence, $H^1(H^{p'})$ is
  isomorphic to the dual of $\vmo(H^p) = [h_{I,J} : I,J\in\mathcal{D}]\subset \bmo(H^p)$ and the
  isomorphism constant between depends just on the isomorphism constant between $H^1(H^{p’})^*$ and
  $\bmo(H^p)$.
\end{remark}

\begin{prop}\label{proposition: predual}
  Let $X$ denote a Banach space with a normalized shrinking basis $(e_j)$ and assume that
  $(e_j^*/\|e_j^*\|)$ is $C$-strategically reproducible in $X^*$.  Then $(e_j)$ is $C$-strategically
  reproducible in $X$.
\end{prop}
 
 \begin{proof}
   For the sake of simplicity, we assume that $(e_j)$ is bimonotone (and thus,
   $\|e_j^*\| = \|e_j\| = 1$); the statement still holds without that assumption and can be proved
   by slightly modifying the argument given below.

   We are now describing a winning strategy for Player II, assuming he has a winning strategy in
   $X^*$.  Assume that in Turn $n$ Step $1$ Player I picks $W_n\in\mathrm{cof}(X)$ and
   $G_n\in\mathrm{cof}_{w^*}(X^*)$.  Using his winning strategy in $X^*$ for the spaces
   $\widetilde G_n = \overline{W_n}^{w*}\in\mathrm{cof}_{w^*}(X^{**})$ and
   $\widetilde W_n = G_n\in\mathrm{cof}(X^*)$, Player II completes Step $2$ of Turn $n$.  Obviously,
   \eqref{enu:strat-rep:i}, \eqref{enu:strat-rep:ii} and~\eqref{enu:strat-rep:iii} are satisfied,
   while~\eqref{enu:strat-rep:iv} follows from the fact that
   $\dist(x_n,\overline{W_n}^{w^*}) = \dist(x_n,W_n)$, which is a consequence of the Hahn-Banach
   theorem.
 \end{proof}

 \begin{remark}\label{remark:predual}
   Using \Cref{proposition: predual}, we are able to deduce the following two assertions.
   By~\cite[Theorem~5.2]{lechner:motakis:mueller:schlumprecht:2018} the Haar basis $(h_I)$ is
   strategically reproducible in $H^1$, hence, $(h_I)$ is also strategically reproducible in $\vmo$.
   Moreover, the biparameter Haar system $(h_{I, J})$ in $\vmo(H^p)$ is $C_p$-strategically
   reproducible for a constant $C_p>0$, which satisfies
   $\sup_{p_0 \leq p \leq p_1} C_p \leq C_{p_0,p_1} < \infty$ whenever
   $1 < p_0 \leq p_1 < \infty$~\cite[Theorem~5.3]{lechner:motakis:mueller:schlumprecht:2018}.
 \end{remark}
 
\begin{defin}\label{Def: asympt curved} Let $X$ be  a Banach space with a basis $(e_n)$. We say that $X$ is {\em asymptotically curved } (with respect to $(e_j)$)
  itf for every normalized block basis $(x_n)$
 $$\lim_{n\to\infty} \frac1n \Big\|\sum_{j=1}^n x_n\Big\|,$$
 
 As already defined in the introduction we call the sequence of Banach spaces $(X_k)$
 \emph{uniformly asymptotically curved} with respect to the array $(e_{(k,j)})$, if for every
 $k\in\mathbb{N}$ and every block basis $(x_{(k,j)})_j$ of $(e_{(k,j)})_j$ we have
 \begin{equation*}
   \lim_n \sup_{k}\Big\|\frac{1}{n}\sum_{j=1}^n x_{(k,j)} \Big\|_{X_k} = 0.
 \end{equation*}
\end{defin}

The following is well known and a special case of Proposition 3 in
\cite{odell:schlumprecht:zsak:2007}, it can also be easily shown directly.

\begin{lem}\label{lem:asymtotically-curved:1}
  Let $X$ denote a Banach space with a Schauder basis $(e_j)$, and let $1 \leq r \leq \infty$,
  $1\leq s \leq \infty$ be such that $\frac{1}{r} + \frac{1}{s} = 1$.  Assume that each block
  sequence $(x_j^*)$ of the coordinate functionals $(e_j^*)$ of $(e_j)$ satisfies the lower
  $r$-estimate
  \begin{equation*}
    \Bigl\|\sum_{j=1}^n x_j^*\Bigr\|_{X^*}
    \geq c \Bigl(\sum_{j=1}^n \|x_j^*\|_{X^*}^r\Bigr)^{1/r},
    \qquad n\in\mathbb{N},
  \end{equation*}
  for some constant $c>0$ independent of $n$.  Then each block sequence $(x_j)$ of $(e_j)$ satisfies
  the upper $s$-estimate
  \begin{equation*}
    \Bigl\|\sum_{j=1}^n x_j\Bigr\|_{X}
    \leq \frac{1}{c} \Bigl(\sum_{j=1}^n \|x_j\|_{X}^s\Bigr)^{1/s},
    \qquad n\in\mathbb{N}.
  \end{equation*}
\end{lem}

The following Lemma is proved easily.
\begin{lem}\label{lem:asymtotically-curved:2}
  Let $1< s < \infty$.  Assume that the array $(e_{k,j})_{k,j}$ is such that $(e_{k,j})_j$ satisfies
  an upper $s$-estimate for each $k$, where the constant $C$ is independent of $k$.  Then the array
  $(e_{k,j})_{k,j}$ is uniformly asymptotically curved.
\end{lem}

\begin{prop}\label{pro:r-estimate}
  Let $1\leq p,q < \infty$.  Then every block sequence of the biparameter Haar system in $H^p(H^q)$
  satisfies the lower $\max(2,p,q)$-estimate with constant $1$ and the upper $\min(2,p,q)$-estimate
  also with constant $1$.
\end{prop}

\begin{proof}
  Before we begin with the actual proof, we define the biparameter square function $\mathbb{S}$ by
  \begin{equation*}
    \mathbb{S}\Bigl( \sum_{I,J\in\mathcal{D}} a_{I,J} h_{I,J}\Bigr)
    = \Bigl(\sum_{I,J\in\mathcal{D}} a_{I,J}^2 h_{I,J}^2\Bigr)^{1/2}.
  \end{equation*}

  Let $(f_i)$ denote a block sequence of the biparameter Haar system.  Note that
  \begin{equation}\label{eq:pro:r-estimate:1}
    \Bigl\|\sum_{j=1}^n f_j\Bigr\|_{H^p(H^q)}
    = \biggl( \int \Bigl(
    \int \Bigl( \sum_{j=1}^n \bigl(\mathbb{S} f_j\bigr)^2 \Bigr)^{q/2}dy
    \Bigr)^{p/q} dx \biggr)^{1/p}.
  \end{equation}

  First, we will show that $H^p(H^q)$ satisfies the upper $\min(2,p,q)$-estimate with constant $1$.
  \begin{description}
  \item[Case $p\geq 2$, $q\geq 2$] Since $q/2\geq 1$, we reinterpret~\eqref{eq:pro:r-estimate:1} and
    use Minkowski's inequality to obtain
    \begin{align*}
      \Bigl\|\sum_{j=1}^n f_j\Bigr\|_{H^p(H^q)}
      &= \biggl( \int
        \Bigl\|\sum_{j=1}^n \bigl(\mathbb{S} f_j\bigr)^2\Bigr\|_{L^{q/2}(y)}^{p/2}
        dx \Biggr)^{1/p}
        \leq \biggl( \int
        \Bigl( \sum_{j=1}^n \bigl\| \bigl(\mathbb{S} f_j\bigr)^2\bigr\|_{L^{q/2}(y)} \Bigr)^{p/2}
        dx \biggr)^{1/p}\\
      &= \biggl( \int
        \Bigl( \sum_{j=1}^n \Bigl( \int \bigl(\mathbb{S} f_j\bigr)^q dy\Bigr)^{2/q}  \Bigr)^{p/2}
        dx \biggr)^{1/p}
        = \Bigl\|\sum_{j=1}^n \Bigl( \int \bigl(\mathbb{S} f_j\bigr)^q
        dy\Bigr)^{2/q}\Bigr\|_{L^{p/2}(x)}^{1/2}\\
      &\leq \biggl( \sum_{j=1}^n
        \Bigl\| \Bigl( \int \bigl(\mathbb{S} f_j\bigr)^q dy\Bigr)^{2/q}\Bigr\|_{L^{p/2}(x)}
        \biggr)^{1/2}
        = \Bigl(\sum_{j=1}^n \|f_j\|_{H^p(H^q)}^2\Bigr)^{1/2}.
    \end{align*}

  \item[Case $p\leq 2$, $q\geq 2$] The first step is the same as in the previous case, \ie, we have
    \begin{equation*}
      \Bigl\|\sum_{j=1}^n f_j\Bigr\|_{H^p(H^q)}
      \leq \biggl( \int
      \Bigl( \sum_{j=1}^n \Bigl( \int \bigl(\mathbb{S} f_j\bigr)^q dy\Bigr)^{2/q}  \Bigr)^{p/2}
      dx \biggr)^{1/p}.
    \end{equation*}
    Since $p/2\leq 1$, we obtain
    \begin{align*}
      \Bigl\|\sum_{j=1}^n f_j\Bigr\|_{H^p(H^q)}
      &\leq \biggl( \int
        \Bigl( \sum_{j=1}^n \Bigl( \int \bigl(\mathbb{S} f_j\bigr)^q dy\Bigr)^{2/q}  \Bigr)^{p/2}
        dx \biggr)^{1/p}\\
      &\leq \biggl( \sum_{j=1}^n
        \int \Bigl( \int \bigl(\mathbb{S} f_j\bigr)^q dy\Bigr)^{p/q}
        dx \biggr)^{1/p}
        = \Bigl( \sum_{j=1}^n \bigl\| f_j\bigr\|_{H^p(H^q)}^p \Bigr)^{1/p}.
    \end{align*}

  \item[Case $p\geq 2$, $q\leq 2$] Since $q/2\leq 1$, \eqref{eq:pro:r-estimate:1} yields
    \begin{align*}
      \Bigl\|\sum_{j=1}^n f_j\Bigr\|_{H^p(H^q)}
      &\leq \biggl( \int \Bigl( \sum_{j=1}^n
        \int \bigl(\mathbb{S} f_j\bigr)^qdy
        \Bigr)^{p/q} dx \biggr)^{1/p}
        = \Bigl\|\sum_{j=1}^n \int \bigl(\mathbb{S} f_j\bigr)^qdy\Bigr\|_{L^{p/q}(x)}^{1/q}\\
      &\leq \biggl( \sum_{j=1}^n \Bigl(
        \int \Bigl( \int \bigl(\mathbb{S} f_j\bigr)^qdy \Bigr)^{p/q} dx
        \Bigr)^{q/p} \biggr)^{1/q}
        =  \Bigl(\sum_{j=1}^n \|f_j\|_{H^p(H^q)}^q\Bigr)^{1/q}.
    \end{align*}

  \item[Case $p\leq 2$, $q\leq 2$] The first step is similar to the previous case, \ie,\
    \begin{equation}\label{eq:pro:r-estimate:2}
      \Bigl\|\sum_{j=1}^n f_j\Bigr\|_{H^p(H^q)}
      \leq \biggl( \int \Bigl( \sum_{j=1}^n
      \int \bigl(\mathbb{S} f_j\bigr)^qdy
      \Bigr)^{p/q} dx \biggr)^{1/p}.
    \end{equation}
    If $p\geq q$, we use Minkowski's inequality in $L^{p/q}$ and obtain
    \begin{equation*}
      \Bigl\|\sum_{j=1}^n f_j\Bigr\|_{H^p(H^q)}
      \leq \biggl\|
      \sum_{j=1}^n \Bigl\| \int \bigl(\mathbb{S} f_j\bigr)^qdy \Bigr\|_{L^{p/q}(x)}
      \biggr\|^{1/q}
      = \Bigl(\sum_{j=1}^n \| f_j\|_{H^p(H^q)}^q\Bigr)^{1/q}.
    \end{equation*}
    If $p\leq q$, \eqref{eq:pro:r-estimate:2} yields
    \begin{equation*}
      \Bigl\|\sum_{j=1}^n f_j\Bigr\|_{H^p(H^q)}
      \leq \biggl(\sum_{j=1}^n \int \Bigl( 
      \int \bigl(\mathbb{S} f_j\bigr)^qdy
      \Bigr)^{p/q} dx \biggr)^{1/p}
      = \Bigl(\sum_{j=1}^n \|f_j\|_{H^p(H^q)}^p \Bigr)^{1/p}.
    \end{equation*}
  \end{description}

  Secondly, we will prove that $H^p(H^q)$ satisfies the lower $\max(2,p,q)$-estimate with constant
  $1$.
  \begin{description}
  \item[Case $p\geq q\geq 2$] Since $q/2\geq 1$, we obtain from~\eqref{eq:pro:r-estimate:1}
    \begin{equation}\label{eq:pro:r-estimate:3}
      \Bigl\|\sum_{j=1}^n f_j\Bigr\|_{H^p(H^q)}
      \geq \biggl( \int \Bigl(
      \sum_{j=1}^n \int \bigl(\mathbb{S} f_j\bigr)^qdy
      \Bigr)^{p/q} dx \biggr)^{1/p}.
    \end{equation}
    Using $p/q\geq 1$ yields
    \begin{equation*}
      \Bigl\|\sum_{j=1}^n f_j\Bigr\|_{H^p(H^q)}
      \geq \biggl( \sum_{j=1}^n \int \Bigl(
      \int \bigl(\mathbb{S} f_j\bigr)^qdy
      \Bigr)^{p/q} dx \biggr)^{1/p}
      = \Bigl( \sum_{j=1}^n\|f_j\|_{H^p(H^q)}^p \Bigr)^{1/p}.
    \end{equation*}

  \item[Case $q\geq p\geq 2$] Using~\eqref{eq:pro:r-estimate:3} and Minkowski's inequality yields
    \begin{align*}
      \Bigl\|\sum_{j=1}^n f_j\Bigr\|_{H^p(H^q)}
      &\geq \biggl( \int \Bigl(
        \sum_{j=1}^n \int \bigl(\mathbb{S} f_j\bigr)^qdy
        \Bigr)^{p/q} dx \biggr)^{1/p}
        = \biggl( \int
        \biggl\|\Bigl(\int \bigl(\mathbb{S} f_j\bigr)^qdy\Bigr)^{p/q}\biggr\|_{\ell^{q/p}(j)}
        dx \biggr)^{1/p}\\
      &\geq \biggl(
        \biggl\|\int \Bigl(\int \bigl(\mathbb{S} f_j\bigr)^qdy\Bigr)^{p/q} dx\biggr\|_{\ell^{q/p}(j)}
        \biggr)^{1/p}
        = \Bigl( \sum_{j=1}^n \|f_j\|_{H^p(H^q)}^q \Bigr)^{1/q}.
    \end{align*}

  \item[Case $p\geq 2\geq q$] By~\eqref{eq:pro:r-estimate:1} and Minkowski's inequality, and since
    $2/q\geq 1$, $p/2\geq 1$ we obtain
    \begin{align*}
      \Bigl\|\sum_{j=1}^n f_j\Bigr\|_{H^p(H^q)}
      &= \biggl( \int \Bigl(
        \int \Bigl\| \bigl(\mathbb{S} f_j\bigr)^{q}\Bigr\|_{\ell^{2/q}(j)} dy
        \Bigr)^{p/q} dx \biggr)^{1/p}
        \geq \biggl( \int
        \Bigl\| \int \bigl(\mathbb{S} f_j\bigr)^{q} dy\Bigr\|_{\ell^{2/q}(j)}^{p/q}
        dx \biggr)^{1/p}\\
      &= \biggl( \int
        \Bigl(\sum_{j=1}^n \Bigl(\int \bigl(\mathbb{S} f_j\bigr)^{q} dy\Bigr)^{2/q}\Bigr)^{p/2}
        dx \biggr)^{1/p}\\
      &\geq \biggl( \sum_{j=1}^n \int
        \Bigl(\int \bigl(\mathbb{S} f_j\bigr)^{q} dy\Bigr)^{p/q}
        dx \biggr)^{1/p}
        = \Bigl( \sum_{j=1}^n \|f_j\|_{H^p(H^q)}^p \Bigr)^{1/p}.
    \end{align*}

  \item[Case $p,q\leq 2$] The first step is the same as in the previous case, \ie,
    \begin{align*}
      \Bigl\|\sum_{j=1}^n f_j\Bigr\|_{H^p(H^q)}
      &\geq \biggl( \int
        \Bigl(\sum_{j=1}^n \Bigl(\int \bigl(\mathbb{S} f_j\bigr)^{q} dy\Bigr)^{2/q}\Bigr)^{p/2}
        dx \biggr)^{1/p}\\
      &= \biggl( \int
        \Bigl\|\Bigl(\int \bigl(\mathbb{S} f_j\bigr)^{q} dy\Bigr)^{p/q}\Bigr\|_{\ell^{2/p}(j)}
        dx \biggr)^{1/p}
    \end{align*}
    Here, $2/p\geq 1$, hence, by Minkowski's inequality, we obtain
    \begin{equation*}
      \Bigl\|\sum_{j=1}^n f_j\Bigr\|_{H^p(H^q)}
      \geq \biggl( 
      \Bigl\| \int\Bigl(\int \bigl(\mathbb{S} f_j\bigr)^{q} dy\Bigr)^{p/q} dx\Bigr\|_{\ell^{2/p}(j)}
      \biggr)^{1/p}
      = \Bigl(\sum_{j=1}^n \|f_j\|_{H^p(H^q)}^2\Bigr)^{1/2}.
    \end{equation*}
  \end{description}
\end{proof}

\begin{remark}\label{remark:Lr(Ls)}
  Let $1 < r,s < \infty$, then the identity operator provides an isomorphism between $H^r(H^s)$ and
  $L^r(L^s)$ (see~Capon~\cite{capon:1982:2}); hence, by \Cref{pro:r-estimate}, each block sequence
  with respect to the biparameter Haar system in $L^r(L^s)$ satisfies an upper $\min(r,s)$-estimate
  with constant $C=C_{r,s}$.  Moreover,
  $\sup_{p_0\leq r,s\leq p_1} C_{r,s}\leq C_{p_0,p_1} < \infty$ whenever
  $1 < p_0 \leq p_1 < \infty$.
\end{remark}

\section{Simultaneous strategical reproducibility and statement of the main results}\label{S:1}

In order to state our main results, we will now state precisely the necessary definitions.  Recall
that we defined $Z$ as the space
\begin{equation}\label{eq:Z-def}
  Z
  =\ell^\infty(X_k:k\in\N)
  =\big\{(x_k) : x_k\in X_k,\,k \kin\N,\, \|(x_k)\|
  =\sup_{k\in\N} \|x_k\|_{X_k} <\infty\big\}.
\end{equation}
We also put
\begin{equation*}
  Y
  =c_0(X_k:k\in\N)
  =\big\{(x_k): x_k\in X_k,\,k\kin\N,\, \lim_{k\to\infty} \|x_k\|_{X_k} =0\big\}.
\end{equation*}
If for some space $X$ we have $X_k=X$, for all $k\in\N$, we will write $\ell^\infty(X)$ and $c_0(X)$
instead of $\ell^\infty(X_k:k\in\N)$ and $c_0(X_k:k\in\N)$.

For $\overline x=(x_k)$ in $Z$ (or $Y$) we call the set $\supp(\overline x)=\{k\in\N: x_k\not=0\}$
{\em the support of $\overline x$ in Z (or $Y$)}.

For $N\subset \N$, and $\overline x=(x_k)\in Z $ we let $P_N(x)\in Z$ be the projection of $x$ on
the coordinates in $N$,\ie,
\begin{equation}\label{eq:PN-def}
  P_N(\overline x)
  = (y_k),
  \text{ with }
  y_k =
  \begin{cases}
    x_k &\text{if $k\in N$, and}\\
    0 &\text{if $k\not\in N$,}
  \end{cases}
\end{equation}
and we put $Z_N:= P_N(Z)$ which is isometrically isomorphically to $\ell^\infty(X_k:k\in N)$ and
will be identified with that space.  For $\overline x=(x_k)\in Z$ and $k\in\N$ we put
$P_k(\overline x)=x_k$.  In particular we identify $X_k$ with its image under the canonical
embedding into $Z$.  We also identify $X^*_k $ in the canonical way as a subspace of $Z^*$
($x^*\in X^*_k$ is acting on the $k$-component of $\bar z=(z_k)\in Z$).  Note that $X^*_k$ is thus a
$w^*$ closed subspace of $Z^*$.  For $k,j\in\N$ we also consider $e_{(k,j)}$ to be an element of
$Z$, and $e^*_{(k,j)}$ to be an element of $Z^*$ in the obvious way.

\begin{convention}\label{conv:1}
  We fix from now on a bijective map $\nu(\cdot, \cdot)\colon \N^2\to \N$, $(k,j)\mapsto \nu(k,j)$
  with the property that for some $i,j,k\in\N$ we have $\nu(k,i)<\nu(k,j)$ if and only if $i<j$.  We
  denote the inverse map by $(\kappa,\iota):\N\to \N^2,\ n\mapsto (\kappa(n), \iota(n))$.  We order
  the array $(e_{(k,j)}: k,j\in\N)$ into a sequence $(e_n)$, by putting
  $e_n=e_{(\kappa(n),\iota(n))}$ and $e^*_n=e_{(\kappa(n),\iota(n))}^*$.  More generally, whenever
  $(x_{(k,j)})_j$ is a sequence in $X_k$, $k\in\mathbb{N}$, then we order the array $(x_{(k,j)})$
  into the sequence $(x_n)$ defined by $x_n=x_{(\kappa(n),\iota(n))}$; we do the same for
  $(x_{(k,j)}^*)$.

  Let $\mathcal{P}$ denote the product topology on $Z$, \ie, the coarsest topology such that all the
  $P_k$, $k\in\mathbb{N}$, are continuous.  Let $\overline{z}^{(j)}\in Z$, $j\in\mathbb{N}$, and
  $\overline z\in Z$.  Then $(\overline z^{(j)})$ converges to $\overline z$ with respect to
  $\mathcal{P}$, if and only if
  \begin{equation*}
    \lim_{j\to\infty}  P_k \overline z^{(j)}= P_k \overline z,\, \text{ for all } k\kin\mathbb{N}.
  \end{equation*}
  Whenever a sequence converges in $Z$, we implicitly refer to convergence with respect to the
  product topology $\mathcal{P}$.  Whenever a sequence converges in some $X_k$, we refer to the norm
  topology in $X_k$.
\end{convention}

\begin{remark}\label{rem:diag-conv}
  For each $k\kin\mathbb{N}$, assume $(e_{(k,j)})_j$ has basis constant $\lambda\geq 1$.  Let
  $\sum_{n=1}^\infty a_n e_n\in Y$, where the series converges in the relative topology
  $\mathcal{P}|_Y$.  Then the series $\sum_{n=1}^\infty a_n e_n$ converges in the norm topology of
  $Y$.
\end{remark}

We now consider the following ``simultaneous version'' of the game described in
\cite{lechner:motakis:mueller:schlumprecht:2018}.
\begin{defin}\label{def:simul-strat-rep}
  Let $C \kge 1$.  We say that the array $(e_{(k,j)})$ is \emph{$C$-simultaneously strategically
    reproducible} if for every $k\in\mathbb{N}$ $(e_{(k,j)})_j$ is $C$-strategically reproducible.
\end{defin}

\begin{remark}\label{rem:2.4}
  Note that we can also describe simultaneously strategically reproducibility in terms of the
  following two player game: The array $(e_{(k,j)})$ is $C\textnormal{-}$ simultaneously
  strategically reproducible if and only if for every $\eta > 0$, Player II has a winning strategy
  for the game $\mathrm{Rep}_{(Z,(e_{(k,j))}))}(C,\eta)$ between Player I and Player II:

  Assume the space $Z$, $P_N$, $(e_n: n\in\N)$ and $(e^*_n: n\in\N)$, are defined as
  in~\eqref{eq:Z-def}, \eqref{eq:PN-def} and in \Cref{conv:1}.
  \begin{description}
  \item[Pregame] Before the first turn Player I is allowed to choose a partitions of
    $\N = N_1\cup N_2$.  For $k\in\N$, and $r=1,2$ let $N_r^{(k)}= \{\nu(k,j):j\in\N\}\cap N_r$.
  \item[Turn $n$, Step\! 1] Player I chooses $\eta_n>0$,
    $G_n\in\mathrm{cof}_{w^*}(X^*_{\kappa(n)})$, and $W_n\in\mathrm{cof}(X_{\kappa(n)})$.
  \item[Turn $n$, Step\! 2] Player II chooses $i_n\in\{1,2\}$, a finite subset $E_n$ of
    $ N^{(\kappa(n))}_{i_n}$ and sequences of non-negative real numbers
    $(\lambda_i^{(n)})_{i\in E_n}$, $(\mu_i^{(n)})_{i\in E_n}$ satisfying
    \begin{equation*}
      1-\eta <\sum_{i\in E_n}\lambda_i^{(n)}\mu_i^{(n)}< 1+\eta.
    \end{equation*}
  \item[Turn $n$, Step\! 3] Player I chooses $(\vp_i^{(n)})_{i\in E_n}$ in $\{-1,1\}^{E_n}$.
  \end{description}
  We say that Player II has a winning strategy in the game $\mathrm{Rep}_{Z,(e_{(k,j)})}(C,\eta)$ if
  he can force the following properties on the result:

  For all $k, j\in\N$ we set $n=\nu(k,j)$ and put
  \begin{equation*}
    x_n=x_{(k,j)} = \sum_{i\in E_n}\vp_i^{(n)} \lambda^{(n)}_ie_{(k,i)}
    \qquad\text{and}\qquad
    x^*_n= x_{(k,j)}^* = \sum_{i\in E_n}\vp_i^{(n)}\mu^{(n)}_ie^*_{(k,i)}
  \end{equation*} and demand:
  \begin{enumerate}[(i)]
  \item\label{enu:simul-game:i} the sequences $(x_{(k,j)})_j$ and $(e_{(k,j)})_j$ are impartially
    $(C+\eta)$-equivalent for each $k\in\N$;
  \item\label{enu:simul-game:ii} the sequences $(x^*_{(k,j)})_j$ and $(e^*_{(k,j)})_j$ are
    impartially $(C+\eta)$-equivalent for each $k\in\N$;
  \item\label{enu:simul-game:iii} for all $n\in\N$ we have $\mathrm{dist}(x^*_n, G_n) < \eta_n$;
  \item\label{enu:simul-game:iv} for all $n\in\N$ we have $\mathrm{dist}(x_n, W_n) < \eta_n$.
  \end{enumerate}
\end{remark}

Completely analogous to \Cref{factorization definitions:basis}, we define the corresponding notions
in $Z$, below.
\begin{defin}\label{defin:Z-prop}
  Let $T\colon Z\to Z$ be an operator.
  \begin{enumerate}[(i)]
  \item We call $T$ a {\em diagonal operator on $Z$}, if $e^*_m(Te_n)=0$, whenever $m\neq n$.
  \item We say that $T$ has {\em a large diagonal} if
    $\inf_{n}\big|e^*_{n}\big(Te_{n}\big)\big| > 0$.
  \end{enumerate}
\end{defin}
We are now in the position to state our two main results.
\begin{thm}\label{T:2.5}
  Assume that there are $C,\lambda\ge 1$, and a map $K\colon (0,\infty)\to (0,\infty)$ so that
  \begin{enumerate}[(i)]
  \item the basis constant of $(e_{(k,j)})_j$, is at most $\lambda$ in $X_k$, for each $k\in\N$;
  \item $(e_{(k,j)})_j$ has the $K$-diagonal factorization property in $X_k$, for each $k\in\N$;
  \item the array $(e_{(k,j)})_{k,j}$ is simultaneously $C$-strategically reproducible in $Z$.
  \end{enumerate}
  
  Let $T\colon Z\to Z$ be bounded and linear, with
  \begin{equation*}
    \delta=\inf_{k,j\in\N} \big| e^*_{(k,j)} (Te_{(k,j)})\big| >0.
  \end{equation*}
  Then for each sequence of infinite subsets $(\Omega_l)$ of $\mathbb{N}$, there is an infinite
  $\Gamma\subset \N$ so that $\Gamma\cap\Omega_l$ is infinite for all $l\in\mathbb{N}$ and the
  identity on $Z_\Gamma$ $\lambda C^{2}K(\delta)$-factors through $T$.
\end{thm}

\begin{remark}
  Note, that we did not simply state in \Cref{T:2.5} that there is an infinite $\Gamma$ so that the
  identity on $Z_\Gamma$ factors through $T$.  However, we show more: additionally we require that
  the intersection of $\Gamma$ with any prespecified sequence of infinite sets $\Omega_l\subset N$,
  $l\in\N$, also has to stay infinite.  This means that if an infinite number of elements of the
  spaces $X_n$ belong to a certain category, then also the spaces in an infinite subset of $\Gamma$
  will belong to that category.  In \Cref{prop:dense} we will provide an application of that
  additional condition on $\Gamma$.
\end{remark}

\begin{cor}\label{C:2.6} Assume that $X$ is a Banach space with a normalized  basis which is $C$-strategically reproducible
  and has the $K$-diagonal factorization property for some $K:(0,\infty)\to(0,\infty)$.  Define
  $e_{(k,j)}$ to be the $j$-th basis element of the $k$-th component in $\ell^\infty(X)$, for
  $k, j\kin\N$.

  Then the array $(e_{(k,j)})_{k,j}$ is simultaneously $C$-strategically reproducible in
  $Z=\ell^\infty(X)$, and the identity on $Z$ factors through every operator $T\colon Z\to Z$ with
  large diagonal.
\end{cor}

The following result describes a situation where it is not necessary, like in \Cref{T:2.5}, to pass
to an infinite subset $\Gamma$ of $\N$.
\begin{thm}\label{T:2.7}
  Assume the array $(e_{(k,j)})$ is uniformly asymptotically curved (see~\eqref{eq:uac}), and
  furthermore, that there are $C,\lambda\ge 1$, and a map $K\colon (0,\infty)\to (0,\infty)$ so that
  \begin{enumerate}[(i)]
  \item the basis constant of $(e_{(k,j)})_j$, is at most $\lambda$ in $X_k$, for each $k\in\N$;
  \item $(e_{(k,j)})_j$ has the $K$-diagonal factorization property in $X_k$, for each $k\in\N$;
  \item the array $(e_{(k,j)})_{k,j}$ is simultaneously $C$-strategically reproducible in $Z$.
  \end{enumerate}
  
  Let $T\colon Z\to Z$ be bounded and linear, with
  \begin{equation*}
    \delta=\inf_{k,j\in\N} \big| e^*_{(k,j)} (Te_{(k,j)})\big| >0.
  \end{equation*}
  Then the identity on $Z$ $\lambda C^{2}K(\delta)$-factors through $T$.
\end{thm}

\section{Factorization through diagonal operators}
\label{sec:fact-thro-diag}

The main purpose of this section is to prove the pivotal \Cref{prop:main}.

\begin{lem}\label{L:2.9}
  Let $S\colon Z\to Z$ be a bounded operator, and let $(\Omega_k)$ denote a sequence of infinite
  subsets of $\mathbb{N}$.  For any $\rho>0$ there is an infinite $\Gamma\subset\N$ so that

  \begin{equation*}
    \Gamma\cap\Omega_k\ \text{is infinite for all $k\in\mathbb{N}$},
  \end{equation*}
  and
  \begin{equation*}
    \big\|P_\Gamma \circ S  |_{Z_\Gamma}\big\|\le 2\big\|S|_{Y}\big\|+\rho.
  \end{equation*}

\end{lem}

\begin{proof} Let $l\in\N$.  We first observe that for a fixed $x^*\in X^*_l$, and for any infinite
  set $\Lambda\subset \N$, such that $\Lambda\cap\Omega_k$ is infinite for all $k\in\mathbb{N}$,
  there is an infinite $\Lambda'\subset\Lambda$ so that $\Lambda'\cap\Omega_k$ is infinite, for all
  $k\in\mathbb{N}$, and $\|x^*\circ P_l\circ S\circ P_{\Lambda'}\|\le \rho/2$.  Indeed, if that were
  not true we could choose for any $n\in\N$ a partition of $\Lambda$ into $n$ infinite subsets
  $\Lambda_1,\Lambda_2,\ldots \Lambda_n$ such that $\Lambda_j\cap\Omega_k$ is infinite for all
  $1\leq j\leq n$, $k\in\mathbb{N}$ and find
  $\overline{x}_{1},\overline{x}_{2}, \ldots \overline{x}_{n}$ in $Z$ with
  $\|\overline{x}_{j}\|\le 1$, $\supp(\overline{x}_{j})\subset \Lambda_j$ and
  $x^* \big(P_l\circ S(\overline{x}_{j})\big)\ge \rho/2$.  But then we would have for
  $\overline{x}=\sum_{j=1}^n \overline{x}_{j}$ that $\|\overline{x}\|\le 1$ and
  $x^* \big(P_l\circ S(\overline{x})\big)\ge n\rho/2$, which is impossible assuming that $n$ is
  chosen large enough.
 
  We now choose a sequence $(y^*_j:j\in\N)$ in $X_l^*$ with $\|y_j^*\|=1$ which norms the elements
  of $X_l$.  Applying our above observation we can choose infinite sets
  $\Lambda_{j+1}\subset\Lambda_j\subset\Lambda$ such that for all $k,j\in\mathbb{N}$
  \begin{equation*}
    \Lambda_j\cap\Omega_k\ \text{is infinite}
    \qquad\text{and}\qquad
    \|y^*_j \circ P_l\circ S\circ P_{\Lambda_j}\|\le \rho/2.
  \end{equation*}
  Then we choose $\Gamma'=\{\lambda_j:j\in\N\}$, where $ \lambda_j\in \Lambda_j$ and
  $\lambda_j<\lambda_{j+1}$, for all $j\in\N$ such that $\Gamma'\cap\Omega_k$ is infinite for all
  $k\in\mathbb{N}$.

  Let $\xb=(x_k)\in Z$, $\|\xb\|\le1 $ and choose $j\in\N$, so that
  \begin{equation*}
    y^*_j\big(P_l\circ S\circ P_{\Gamma'} (\xb)\big)>\frac12\|P_l\circ S\circ P_{\Gamma'}(\xb)\|.
  \end{equation*}

  Then
  \begin{align*}
    \big\|P_l\circ S\circ P_{\Gamma'}(\xb)\big\|
    &\le \big\|P_l\circ S\circ P_{\{\lambda_1,\lambda_2,\ldots ,\lambda_{j-1}\}}(\xb)\big\|
      + \big\|P_l\circ S\circ P_{ \Lambda_j}(\xb)\big\|\\
    &\le\|P_l\circ S |_{Y}\big\| + 2y^*_j(P_l\circ S\circ P_{\Lambda_j} (\xb))
      \le \| S |_{Y}\big\| +\rho.
  \end{align*}
  Thus, we deduce that
  \begin{equation*}
    \big\|P_l\circ S|_{Z_{\Gamma'}}\|\le \| S |_{Y}\big\| +\rho.
  \end{equation*}
 
  Let $(k_j)\subset \N$ be a sequence in which every $k\kin \N$ appears infinitely often.  Starting
  by letting $\Gamma_0=\N$ and $\gamma_1=1$, we can apply our observation and recursively choose
  infinite sets $\Gamma_0\supset \Gamma_1\supset \Gamma_2\supset\ldots $, and
  $\gamma_1<\gamma_2<\ldots $ so that
  \begin{equation*}
    \Gamma_j\cap\Omega_k\ \text{is infinite for all $j,k\in\mathbb{N}$,}\, \gamma_j\in \Gamma_{j-1}\cap \Omega_{k_j}
    \text{ and }
    \big\|P_{\gamma_j}\circ S|_{Z_{\Gamma_j}}\|
    \leq \| S |_{Y}\big\| +\rho.
  \end{equation*}
  
  Finally, letting $\Gamma=\{\gamma_j: j\in\N\}$, we deduce that
  \begin{align*}
    \|P_\Gamma S|_{Z_{\Gamma}}\|
    &= \sup_{j\in\N} \|P_{\gamma_j} S|_{Z_{\Gamma}}\|\\
    &\le \sup_{j\in\N} \Big(
      \big\|P_{\gamma_j} S|_{Z_{\{\gamma_1,\gamma_2,\ldots\gamma_j\}}}\big\|
      + \big\|P_{\gamma_j} S|_{Z_{\Gamma_j}}\big\|
      \Big)
      \le 2\big\| S|_{Y}\big\|
      +\rho,
  \end{align*}
  which proves our claim.
\end{proof}

\begin{lem}\label{L:2.10}
  We assume that the array $(e_{(k,j)})_{k,j}$ is uniformly asymptotically curved.  Let
  $z^*\kin Z^*$ and $\eta>0$.  Then there exists $(m_k)\subset\N$ so that for every $w=(w_k)\in Z$,
  $\|w\|\le 1$ with $w_k\in[e_{(k,j)}: j\ge m_k]$, $k\in\N$, it follows that $|z^*(w)|\le\eta $.  In
  other words, letting $W_k=[e_{(k,j)}:j\ge m_k]$ it follows that
  \begin{equation*}
    \|z^*|_{\ell^\infty(W_k:k\in\mathbb{N})}\|\le \eta.
  \end{equation*}
\end{lem}

\begin{proof} Assume that our claim is not true for some $z^*\in Z^*$ and $\eta>0$.  Then we can
  choose inductively for every $n\in\N$, sequences $(m^{(n)}_k)_{k\in\N}\subset \N$, and
  $\overline{z}_{n} =(w^{(n)}_k)_k\in B_Z$, so that
  \begin{align}
    \label{E:2.10.1} &m^{(n-1)}_k< m^{(n)}_k\text{ for all $k\in\N$ (with $m^{(0)}_k=0$)},\\
    \label{E:2.10.2}  &w^{(n)}_k\in[e_{(k,j)}:m^{(n-1)}_k< j\le m^{(n)}_k],\text{ for all $k\in\N$, and}\\
    \label{E:2.10.3} &z^*( (w^{(n)}_k)_k)>\eta/2.
  \end{align}
  Then for $n\in\N$ define
  \begin{equation*}
    \overline{u}_{n}= \Big(\frac{1}{n}\sum_{m=1}^n w^{(m)}_k:k\in\N\Big)\in Z,
  \end{equation*}
  It follows from our assumption on the spaces $X_k$ that $\lim_n\|\overline{u}_{n}\|_Z=0$, but on
  the other hand we have
  \begin{equation*}
    z^*\big(\overline{u}_{n}\big)
    = \frac{1}{n}\sum_{m=1}^n z^*\bigl((w^{(n)}_k)_k\bigr)
    > \eta/2
  \end{equation*}
  which for large enough $n$ leads to a contradiction.
\end{proof}

Let $(e_j)$ denote a basic sequence in a Banach space $X$.  We say that a sequence $(x_j)$ in $X$ is
a \emph{perturbation of a block basic sequence of $(e_j)$} if there exists a block basis sequence
$(\widetilde x_j)$ of $(e_j)$ such that $\sum_{j=1}^\infty\|x_{j} - \widetilde x_{j}\|_X < \infty$.
\begin{notation}\label{ntn:basic-operators}
  Let $\lambda\ge 1$ and assume the basis constant of $(e_{(k,j)})_j$, is not larger than $\lambda$,
  for all $k\in\mathbb{N}$.  Assume that for each $k\in\mathbb{N}$, $(x_{(k,j)})_j$ is a sequence in
  $X_k$ and that $(x_{(k,j)}^*)_j$ is a perturbation of block basic sequence of $(e_{(k,j)})_j$ in
  $X_k^*$.  Moreover, assume that
  \begin{enumerate}[(i)]
  \item $(x_{(k,j)})_j$ and $(e_{(k,j)})_j$ are impartially $C$-equivalent, for all $k\in\N$;
  \item $(x^*_{(k,j)})_j$ and $(e^*_{(k,j)})_j$ are impartially $C$-equivalent, for all $k\in\N$;
  \item $1-\eta< x^*_{(k,i)}(x_{(k,i)}) < 1+\eta$, for all $k,i\in\N$.
  \end{enumerate}
  Then for each $k,j\kin\N$, we define
  \begin{align*}
    &A_k\colon X_k\to X_k,
    &A_ke_{(k,j)}
    &= x_{(k,j)},\\
    &B_k\colon X_k\to X_k,
    &B_k x
    &= \sum_{j=1}^\infty x_{(k,j)}^*(x)e_{(k,j)}.
  \end{align*}
  and their respective vector operator version
  \begin{align*}
    &A\colon Z\to Z,
    &A\big((z_{k})_k\big)
    &=\big(A_kz_{k}\big)_k,\\
    &B\colon Z\to Z,
    &B\big((z_{k})_k\big)
    &=\big(B_kz_{k}\big)_k.
  \end{align*}
\end{notation}

\begin{remark}\label{remark:basic-operators}
  In view of our hypothesis, the operators $A_k$, $B_k$, $k\in\mathbb{N}$ and consequently $A$, $B$
  in \Cref{ntn:basic-operators} are well defined and satisfy $\|A\|\leq \sqrt C$ and
  $\|B\|\leq \lambda C$ (see~\cite[Lemma~3.7,
  Lemma~3.14]{lechner:motakis:mueller:schlumprecht:2018}).  Moreover, for each $k,i\in\N$ we have
  \begin{equation}\label{eq:basic-operators:1}
    BTA(e_{(k,i)})
    = \Big(\sum_{j=1}^\infty x_{(l,j)}^*(Tx_{(k,i)})e_{(l,j)}:l\in\N\Big)\in Z
  \end{equation}
  it follows (infinite sums are meant to converge with respect to the topology $\mathcal{P}$,
  introduced in Convention \ref{conv:1}) for an $x=\sum_{n=1}^\infty a_n e_n\in Z$ that
  \begin{equation}\label{eq:basic-operators:2}
    \begin{aligned}
      BTA\Bigl(\sum_{n=1}^\infty a_n e_n\Bigr)
      &=\sum_{m=1}^\infty x^*_m\Big(TA\sum_{n=1}^\infty a_n e_n\Big)e_m\\
      &= \sum_{m=1}^\infty \sum_{n=1}^{m-1} a_n x_m^*(Tx_n) e_m
      + \sum_{m=1}^\infty a_m x_m^*(T x_m) e_m\\
      &\qquad\qquad+ \sum_{m=1}^\infty x_m^*\bigl( T \sum_{n=m+1}^\infty a_n x_n \bigr) e_m.
    \end{aligned}
  \end{equation}
  (Note that $T$ might not be continuous with respect to $\mathcal P$, we only used the linearity of
  $T$).  If $x= \sum_{n=1}^{\infty} a_n e_n\in Y$ (which implies that this series is norm
  convergent) then
  \begin{equation}\label{eq:basic-operators:3} BTA\Bigl(\sum_{n=1}^\infty a_n e_n\Bigr)
    =\sum_{m=1}^\infty  \sum_{n=1}^\infty a_n x_m^*\bigl( Tx_n \bigr) e_m. \end{equation}
\end{remark}

We now formulate and prove a rather technical proposition which presents the heart of the proof of
\Cref{T:2.5} and \Cref{T:2.7}.
\begin{prop}\label{prop:main}
  Assume that for some $\lambda\ge 1$ the basis constant of $(e_{(k,j)})_j$, is not larger than
  $\lambda$.  Let $T\colon Z\to Z$ be a bounded linear operator, for each $k\kin\N$ let
  $(x_{(k,j)})_j$ be a sequence in $X_k$ and let $(x^*_{(k,j)})_j$ be a perturbation of a block
  basis of $(e_{(k,j)}^*)_j$ in $X_k^*$.
 
  Let $0 < \eta\leq 1$, $C\ge 1$ and $(\eta_n)\subset(0,1]$ so that
  $\sum_{m=1}^\infty \sum_{n=m+1}^\infty \eta_n<\eta$.  Consider the following conditions:
  \begin{enumerate}[(i)]
  \item\label{prop:i} $(x_{(k,j)})_j$ and $(e_{(k,j)})_j$ are impartially $C$-equivalent, for all
    $k\in\N$;
  \item\label{prop:ii} $(x^*_{(k,j)})_j$ and $(e^*_{(k,j)})_j$ are impartially $C$-equivalent, for
    all $k\in\N$;
  \item\label{prop:iii} $1-\eta< x^*_n(x_n)<1+\eta$, $n\in\mathbb{N}$;
  \item\label{prop:iv} $\sum_{n=1}^{m-1} |x^*_m(Tx_n)| < \eta_m$, for all $m\in\mathbb{N}$;
  \item\label{prop:v} $\sum_{n=m+1}^\infty |x^*_m(Tx_n)| < \eta_m$, for all $m\in\mathbb{N}$;
  \end{enumerate}

  In order to formulate the last condition, we assume that for each $n\in\mathbb{N}$, we are given a
  sequence $(W^{(n)}_k)_k$, where $W^{(n)}_k$ is a cofinite dimensional subspace of $X_k$, with
  $W^{(n+1)}_k\subset W_k^{(n)}$ for $k,n\in\mathbb{N}$.
  \begin{enumerate}[(i)]
    \setcounter{enumi}{5}
  \item\label{prop:v-alt} For each $n\in\mathbb{N}$, assume that
    \begin{equation*}
      \|T^*x_{n}^*|_{\ell^\infty(W_k^{(n+1)} : k\in\N)}\| < \eta_{n}
      \qquad\text{and}\qquad
      \dist\big(x_n, W_{\kappa(n)}^{(n)}\big) < \eta_n,
    \end{equation*}
    for all $n\in\mathbb{N}$.
  \end{enumerate}
  Let $D\colon Z \to Z$ denote the diagonal operator given by
  \begin{equation*}
    De_{(k,j)}
    = x_{(k,j)}^*(Tx_{(k,j)})e_{(k,j)},
    \qquad \text{for all}\ k,j\in\mathbb{N}.
  \end{equation*}
  Then the following assertions~\eqref{prop:a} and~\eqref{prop:b} hold true.
  \begin{enumerate}[(a)]
  \item\label{prop:a} If~\eqref{prop:i}--\eqref{prop:v} is satisfied, then $BTA - D\colon Y\to Y$ is
    well defined and
    \begin{equation*}
      \|BTA - D: Y\to Y\|
      \leq 2\lambda\eta.
    \end{equation*}
    $D$ is a bounded operator from $Z$ to $Z$ and thus also $D-BTA$.  Moreover, for each sequence of
    infinite subsets $(\Omega_l)$ of $\mathbb{N}$, there exists an infinite set
    $\Gamma\subset\mathbb{N}$ such that
    \begin{equation*}
      \Gamma\cap\Omega_l\ \text{is infinite, for all $l\in\mathbb{N}$,}
      \text{ and }
      \|P_\Gamma D - P_\Gamma BTA: Z_\Gamma\to Z_\Gamma\| < 5\lambda\eta.
    \end{equation*}
    If we additionally assume that $K\geq 1$ is such that the identity $K$-factors through
    $P_\Gamma D|_{Z_\Gamma}$ and $\eta < 1/(5\lambda K)$, then the identity on $Z_\Gamma$
    $\Big(\frac{\lambda KC^{2}}{1-5\lambda K\eta}\Big)$-factors through $T$.

  \item\label{prop:b} If alternatively, \eqref{prop:i}--\eqref{prop:iv} and~\eqref{prop:v-alt} are
    satisfied, then
    \begin{equation*}
      \|BTA - D\colon Z\to Z\| < 2\lambda \sqrt{C} (3 + \|T\|).
    \end{equation*}
    If we additionally assume that $K\geq 1$ is such that the identity on $Z$ $K$-factors through
    $D$ and $\eta < 1/(2\lambda \sqrt{C} (3 + \|T\|)K)$ then the identity on $Z$
    $\Big(\frac{\lambda KC^{2}}{1 - 2\lambda \sqrt C (3+\|T\|)K\eta}\Big)$-factors through $T$.
  \end{enumerate}
\end{prop}

\begin{proof}
  Naturally, the proof splits into two parts.
  \begin{proofstep}[Proof of~\eqref{prop:a}]
    Let $y = \sum_{n=1}^\infty a_n e_n\in Y$ be a norm convergent series and observe that
    by~\eqref{eq:basic-operators:3}
    \begin{equation}\label{eq:BTA-D:estimate}
      (BTA-D)y
      = \sum_{m=1}^\infty \sum_{n:n\neq m} a_n x_m^*(T x_n) e_m.
    \end{equation}
    Hence, for each $k\in\mathbb{N}$ we obtain
    \begin{equation*}
      P_k(BTA-D)y
      = \sum_{m:\kappa(m)=k} \sum_{n:n\neq m} a_n x_m^*(T x_n) e_m,
    \end{equation*}
    and thus
    \begin{equation}\label{eq:pk-estimate}
      \begin{aligned}
        \|P_k(BTA-D)y\|_{X_k}
        &\leq \sum_{n=1}^\infty |a_n| \sum_{m:\kappa(m)=k}  |x_m^*(T x_n)|\\
        &\leq 2\lambda \|y\|_Y \sum_{n=1}^\infty \sum_{m:\kappa(m)=k} |x_m^*(T x_n)|.
      \end{aligned}
    \end{equation}
    By~\eqref{prop:iv} and~\eqref{prop:v} we have $\sum_{m,n:m\neq n} |x_m^*(T x_n)| < \infty$, and
    hence
    \begin{equation*}
      \lim_{k\to \infty} \sum_{n=1}^\infty \sum_{m:\kappa(m)=k} |x_m^*(T x_n)|
      = 0.
    \end{equation*}
    Together with estimate~\eqref{eq:pk-estimate} we obtain $(BTA-D)y\in Y$.

    Using~\eqref{prop:iv} and \eqref{prop:v}, ~\eqref{eq:BTA-D:estimate} yields
    \begin{align*}
      \|(BTA-D)y\|_Y
      &\leq \sum_{m=1}^\infty \sum_{n=1}^{m-1} |a_n| |x_m^*(T x_n)|
        + \sum_{m=1}^\infty \sum_{n=m+1}^N  |a_n| |x_m^*( T x_n )|
        \leq 2 \lambda \eta \|y\|_Y.
    \end{align*}
    We conclude
    \begin{equation*}
      \|BTA-D\colon Y\to Y\|
      \leq 2 \lambda \eta.
    \end{equation*}

    Next, we observe that since $D$ is a diagonal operator, which means that $D$ is of the form
    \begin{equation*}
      D\colon Y\to Y, \quad (x_k)\mapsto (D_kx_k),
    \end{equation*}
    where the operators $D_k\colon X_k\to X_k$, $k\in\N$, are uniformly bounded.  Thus, $D$ is also
    a bounded operator on $Z$ with norm $\sup_k \|D_k\colon X_k\to X_k\|$ and the operator $BTA-D$
    is a well defined and a bounded operator on all of $Z$.  Our conclusion follows therefore from
    \Cref{L:2.9} for some infinite set $\Gamma\subset\N$.

    For the additional part, assume that $\hat B\colon Z_\Gamma \to Z_\Gamma$ and
    $\hat A\colon Z_\Gamma\to Z_\Gamma $ are such that $\|\hat B\|\|\hat A\| \leq K$ and
    $I = \hat BD_\Gamma\hat A$.  It follows that
    $\|I - \hat BBTA\hat A\|= \|\hat B(D_\Gamma-BTA)\hat A\| < 5\lambda K\eta < 1$.  Hence, the map
    $Q = \hat BBTA\hat A$ is invertible with $\|Q^{-1}\| \leq 1/(1 - 5\lambda K\eta)$.  In
    conclusion, if we set $\tilde B = Q^{-1}\hat BB$, $\tilde A = A\hat A$ then
    $\tilde BT\tilde A = I$ and $\|\tilde B\|\|\tilde A\| \leq \lambda K C^{2}/(1-5\lambda K\eta)$.
  \end{proofstep}

  \begin{proofstep}[Proof of~\eqref{prop:b}]
    Note that since $A$ is a bounded operator (see \Cref{remark:basic-operators}), it follows that
    \begin{equation*}
      A(z)=\Big(\sum_{j=1}^\infty a_{\nu(k,j)} x_{(k,j)}:k\in\N\Big)= \sum_{n=1}^\infty a_n  x_n,
    \end{equation*}
    whenever $z=\sum_{n=1}^\infty a_n e_n\in Z$, where the series converges in the product topology
    $\mathcal{P}$.  So, in particular the sum $\sum_{n=1}a_n x_n$ is well defined in $Z$ if the sum
    $\sum_{n=1}a_n e_n$ is well defined. Let us assume that $z=\sum_{n=1}^\infty a_n e_n\in S_Z$ and
    thus $\|A(z)\|\le \sqrt C$.

    By~\eqref{eq:basic-operators:2} and the definition of $D$ we obtain
    \begin{equation*}
      (BTA - D)\Bigl(\sum_{n=1}^\infty a_n e_n\Bigr)
      = \sum_{m=1}^\infty \sum_{n=1}^{m-1} a_n x_m^*(T x_n) e_m
      + \sum_{m=1}^\infty x_m^*\bigl( T \sum_{n=m+1}^\infty a_n x_n \bigr) e_m.
    \end{equation*}
    The norm of the first sum is dominated by
    \begin{equation}\label{E:4.5.1}
      \sum_{m=1}^\infty \sum_{n=1}^{m-1} |a_n| |x_m^*(T x_n)|
      < 2\lambda \sum_{m=1}^\infty \eta_m
      \leq 2\lambda \eta.
    \end{equation}
    
    To estimate the norm of the second sum, first choose $y_n\in W_{\kappa(n)}^{(n)}$ according
    to~\eqref{prop:v-alt} such that $\|x_n - y_n\| < \eta_n$.  Let $m\in\mathbb{N}$ be fixed.  We
    claim that $y = \sum_{n=m+1}^\infty a_n y_n$ is a well defined element in
    $\ell^\infty\bigl(W_k^{(m+1)}:k\kin\N\bigr)$. Indeed, it is well defined since
    $\sum_{n=m+1}^\infty a_n x_n\in Z$ and $\sum_{n=1}^\infty\eta_n < \eta < \infty$.  Moreover, by
    the properties of our enumeration (see~\Cref{conv:1}) and since $W_k^{(n+1)}\subset W_k^{(n)}$,
    $k,n\in\mathbb{N}$, we have that
    \begin{equation*}
      y
      = \sum_{n=m+1}^\infty a_n y_n
      = \sum_{k=1}^\infty \sum_{\substack{n > m\\\kappa(n)=k}} a_n y_n
      \in \ell^\infty\bigl(W_k^{(n_k)}:k\kin\N\bigr),
    \end{equation*}
    where $n_k = \min\{ n > m : \kappa(n) = k\} \geq m+1$; thus we proved
    $y\in \ell^\infty\bigl(W_k^{(m+1)}:k\kin\N\bigr)$.  By a standard perturbation argument, we
    obtain
    \begin{align*}
      \sum_{m=1}^\infty \bigl|x_m^*\bigl( T \sum_{n=m+1}^\infty a_n x_n \bigr)\bigr|
      &\leq \sum_{m=1}^\infty \bigl|T^*x_m^*\bigl(\sum_{n=m+1}^\infty a_n y_n \bigr)\bigr|
        + \sum_{m=1}^\infty \bigl|T^*x_m^*\bigl(\sum_{n=m+1}^\infty a_n (x_n - y_n) \bigr)\bigr|\\
      &\leq \sum_{m=1}^\infty \eta_m \Bigl\| \sum_{n=m+1}^\infty a_n y_n\Bigr\|
        + \sum_{m=1}^\infty \|T^*x_m^*\| \sum_{n=m+1}^\infty |a_n| \eta_n\text{\ \  (by (vi))}\\
      &\leq \sum_{m=1}^\infty \eta_m (2\lambda \sqrt C+2\lambda \eta)
        + \sum_{m=1}^\infty \|T\| \sqrt{C} 2\lambda  \sum_{n=m+1}^\infty \eta_n\\
      &\leq \eta 2\lambda \sqrt{C} (2 + \|T\|) 
    \end{align*}
    which together with \eqref{E:4.5.1} establishes the first part of (b).
 
    For the additional part, assume that $\hat B\colon Z \to Z$ and $\hat A\colon Z\to Z $ are such
    that $\|\hat B\|\|\hat A\| \leq K$ and $I = \hat BD\hat A$.  It follows that
    $\|I - \hat BBTA\hat A\|= \|\hat B(D-BTA)\hat A\| < 2\lambda \sqrt C (3+\|T\|)K\eta < 1$.
    Hence, the map $Q = \hat BBTA\hat A$ is invertible with
    $\|Q^{-1}\| \leq 1/(1 - 2\lambda \sqrt C (3+\|T\|)K\eta)$.  In conclusion, if we set
    $\tilde B = Q^{-1}\hat BB$, $\tilde A = A\hat A$ then $\tilde BT\tilde A = I$ and
    $\|\tilde B\|\|\tilde A\| \leq \lambda K C^{2}/(1 - 2\lambda \sqrt C (3+\|T\|)K\eta)$.\qedhere
  \end{proofstep}
\end{proof}

%%%%%%%%%%%%%%%%%%%%%%%%%%%%%%
We finally prepare the work of the next section by isolating the following technical \Cref{L:2.8}.
\begin{lem}\label{L:2.8}
  Assume that $\lambda\ge 1$ and $K\colon (0,\infty)\to (0,\infty)$ is continuous, so that for each
  $k$ the basis constant of $(e_{(k,j)})_j$, is not larger than $\lambda$ and which has the
  $K(\delta)$-factorization property in $X_k$.
        
  Then for every bounded diagonal operator $T\colon Z\to Z$, for which
  \begin{equation*}
    \delta=\inf_{k,j\in\N} \big| e^*_{(k,j)} (Te_{(k,j)})\big|>0,
  \end{equation*}
  the identity almost $K(\delta)$-factors through $T$.
\end{lem}
   
\begin{proof}
  Let $D\colon Z\to Z$ be a bounded diagonal operator with
  \begin{equation*}
    \delta=\inf_{k,j\in\N} \big|e^*_{(k,j)}\big( Te_{(k,j)}\big)\big|>0.
  \end{equation*}
  For $k\in\N$ let $D_k=D|_{X_k}$.  Then $D_k$ is a diagonal operator from $X_k$ to $X_k$.  Next,
  let $\eta>0$.  Then there are for each $k\in\N$ bounded operators $A _k\colon X_k\to X_k$ and
  $B _k\colon X_k\to X_k$, so that $\|A_k\|\cdot \|B_k\|\le K(\delta) + \eta$ and
  $I_k =B_k \circ D_k\circ A_k$, where $I_k$ is the identity on $X_k$.  We can assume that
  $\|A_k\|=1$ and that $\|B_k\|\le K(\delta)+\eta$.  Putting
  \begin{align*}
    A&\colon Z\to Z, \quad (z_{k})\mapsto \big(A_kz_{k}\big),\\
    B&\colon Z\to Z, \quad (z_{k})\mapsto \big(B_kz_{k}\big),
  \end{align*}
  it follows that $\|A\|=1$ and $\|B\|\le K(\delta)+\eta$, and $I_Z=B\circ D\circ A$.
\end{proof}

\section{Proof of \Cref{T:2.5} and \Cref{T:2.7}}
\label{sec:proof-creft:2.5-cref}

The proof of both theorems \Cref{T:2.5} and \Cref{T:2.7} is organized as depicted in the flowchart

\Cref{fig:chart}.
\begin{figure}[bth]
  \includegraphics[scale=0.25]{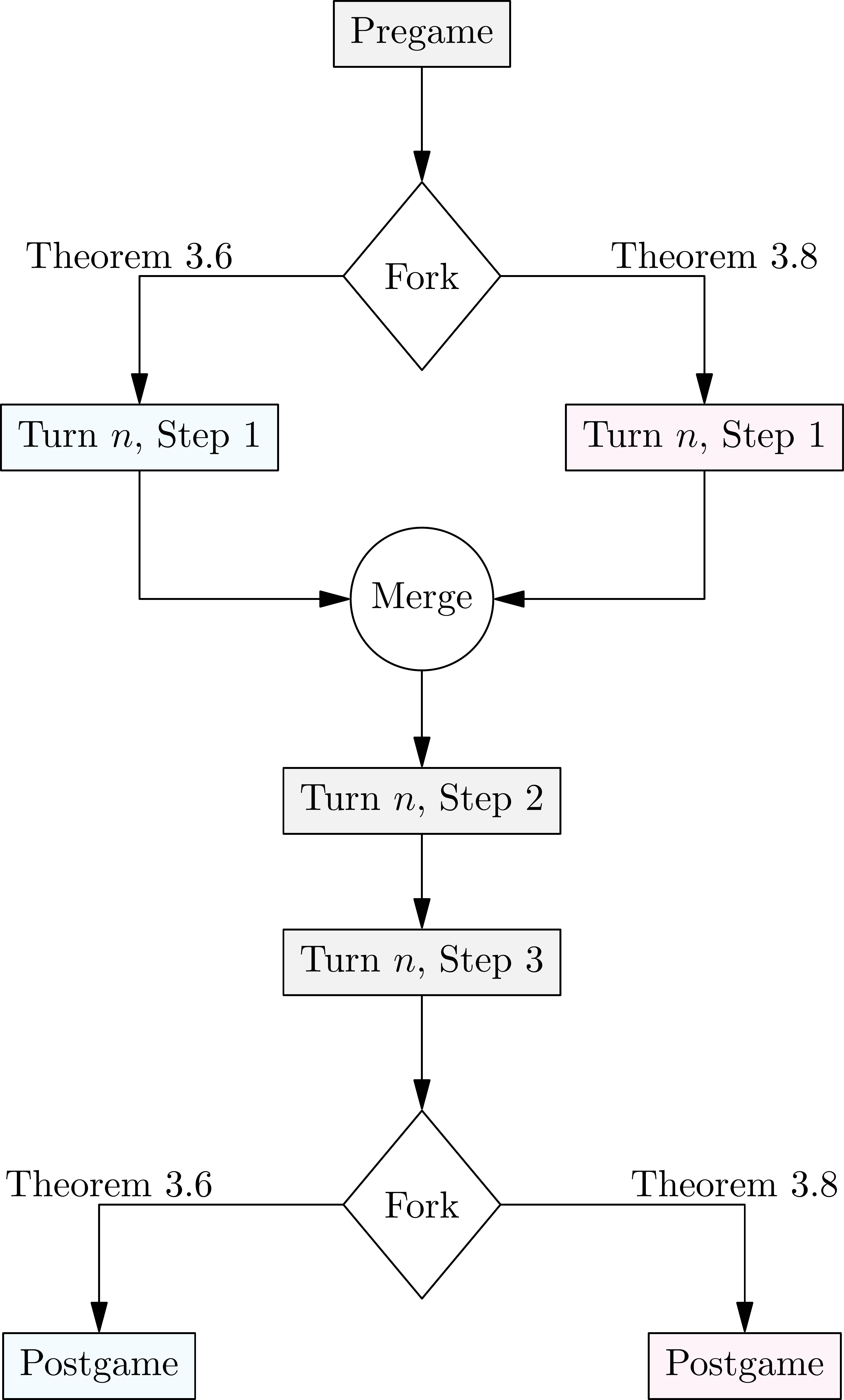}
  \caption{Flowchart of the proof of \Cref{T:2.5} and \Cref{T:2.7}.}
  \label{fig:chart}
\end{figure}
According to the flowchart \Cref{fig:chart}, their respective proofs deviate at two distinct points.
In the proof below, we indicate that, by arranging the text side by side in separate columns.  We
use a corresponding color scheme in the flowchart and in the proof.

Let $T\colon Z\to Z$ be a bounded linear operator with
$\inf_{k,j\in\N} \big|e_{(k,j)}^*(Te_{(k,j)})\big| =: \delta>0$.  By \Cref{rem:2.4}, Player II has a
winning strategy in the game $\mathrm{Rep}_{(Z,(e_{(k,j)}))}(C,\eta)$.  Let us fix $\eta>0$ to be
determined later.
  
We will now describe a strategy for Player I in a game $\mathrm{Rep}_{(X,(e_i))}(C,\eta)$, and
assume Player II answers by following his winning strategy.

\subsubsection*{Pregame}
\label{sec:pregame}
At the beginning, Player I chooses $N_1 = \{n\in\N:e_n^*(Te_n) \geq \delta\}$ and
$N_2 = \{n\in\N: e_n^*(Te_n) \leq -\delta\}$.  For $k\in\N$, and $i=1,2$ let
$N_i^{(k)}= \{\nu(k,j):j\in\N\}\cap N_i$.

\subsubsection*{Turn $n$, Step $1$}
\label{sec:turn-n-step-1}

During the $n$-th turn, Player I proceeds as follows.  In the first step of the $n$-th turn he
chooses
\begin{equation}\label{eq:eta_n}
  \eta_n
  < \frac{\eta}{2^{n+2} n (1+\|T\|) \sqrt{C+\eta}}.
\end{equation}
Let $l_n =\max (\bigcup_{m=1}^{n-1} E_m)$ if $n>1$ (where the finite sets $E_1$, $E_2$,...$E_{n-1}$
were chosen in previous turns and $l_1=1$).  Define
\begin{equation*}
  A_n
  = \big\{ x_m,Tx_m : m < n\} \cup \{ e_{(\kappa(n),i)}, Te_{(\kappa(n),i)}: i\le l_n\}
\end{equation*}
and put
\begin{equation}\label{eq:G_n}
  G_n = A_n^\perp\cap X^*_{\kappa(n)}
\end{equation}
as subspace of $X^*_{\kappa(n)}$ and $Z^*$.

\bigskip
\begin{mdframed}[backgroundcolor=black!5,leftline=false,rightline=false,topline=false,bottomline=false]
  \begin{minipage}[t]{.5\linewidth}
    \begin{mdframed}[backgroundcolor=cyan!5]
      \begin{mdframed}[backgroundcolor=gray!20]
        Choice of $W_n$ in \Cref{T:2.5}.
      \end{mdframed}
      \hfill\newline Now, let
      \begin{align*}
        B_n
        &= \big\{ x^*_m,T^*x_m^*:m < n\big\}\\
        &\quad\cup \{ e^*_{(\kappa(n),i)},T^*e^*_{(\kappa(n),i)}: i\le l_n\}
      \end{align*}
      and define
      \begin{equation}\label{eq:V_n}
        W_n = (B_n)_\perp\cap X_{\kappa(n)}
      \end{equation}
      as a subspace of $X_{\kappa(n)}$ and $Z$.
    \end{mdframed}
  \end{minipage}
  \begin{minipage}[t]{.5\linewidth}
    \begin{mdframed}[backgroundcolor=magenta!5]
      \begin{mdframed}[backgroundcolor=gray!20]
        Choice of $W_n$ in \Cref{T:2.7}.
      \end{mdframed}
      \hfill\newline In order to choose the required cofinite dimensional subspace $W_n$ of
      $X_{\kappa(n)}$ Player I also chooses at the $n$-th step a sequence $(V^{(n)}_{l})_l$, where
      $V_l^{(n)} $ is a finite codimensional subspace of $X_{l}$ of the form
      \begin{equation}\label{eq:V_l^n}
        V_l^{(n)}
        =[e_{(l,j)}:j\ge m(n,l)],
      \end{equation}
      with $V^{(n)}_{l}\subset V^{(m)}_{l}$ for $l\in\N$, $0\le m<n$ and we put $V^{(0)}_l=X_l$ for
      $l\in\N$.  By using \Cref{L:2.10} finitely many times Player I can find for each $l$ a
      subspace $V^{(n)}_l$ of $V^{(n-1)}_l$ so that for all $j<n$
      \begin{equation}\label{E:2.5.3}
        \|T^*x^*_j|_{\ell^\infty(V^{(n)}_l:l\in\N)}\|\le \eta_n
      \end{equation}
      and he chooses
      \begin{equation}\label{eq:W_n-alt}
        W_n\keq V^{(n)}_{\kappa(n)}\keq[e_{(\kappa(n),j)} :j\kge m_n],
      \end{equation}
      where $m_n=m(n,\kappa(n))$.
    \end{mdframed}
  \end{minipage}
\end{mdframed}

\subsubsection*{Turn $n$, Step $2$}
\label{sec:turn-n-step-2}

Player II, following a winning strategy, chooses $i_n\in\{1,2\}$, picks a finite set
$E_n\subset N^{(\kappa(n))}_{i_n}$ and sequences of non-negative scalars
$(\lambda_i^{(n)})_{i\in E_n}$, $(\mu_i^{(n)})_{i\in E_n}$ with
\begin{equation*}
  1-\eta
  < \sum_{i\in E_n}\lambda_i^{(n)}\mu_i^{(n)}
  < 1+\eta.
\end{equation*}

\subsubsection*{Turn $n$, Step $3$}
\label{sec:turn-n-step-3}

Then Player I picks signs $(\varepsilon_i^{(n)})_{i\in E_n}\in\{-1,+1\}^{E_n}$ so that whenever
\begin{equation*}
  x_n = \sum_{i\in E_n}\lambda_i^{(n)}\varepsilon_i^{(n)}e_{(\kappa(n),i)}
  \qquad\text{and}\qquad
  x_n^* = \sum_{i\in E_n}\mu_i^{(n)}\varepsilon_i^{(n)}e_{(\kappa(n),i)}^*,
\end{equation*}
we have
\begin{equation}\label{E:2.5.1}
  |x_n^*\bigl(T x_n\bigr)| > (1-\eta)\delta.
\end{equation}
That it is possible to choose such signs $(\varepsilon_i^{(n)})_{i\in E_n}\in\{-1,+1\}^{E_n}$
follows from the following probabilistic argument: Let $r=(r_j)_{j\in E_n}$ be a Rademacher
sequence, meaning that $r_j$, $j\in E_n$, are independent random variables on some probability space
$(\Omega,\Sigma,\P)$, with $\P(r_j\keq1)=\P(r_j\keq-1)=\frac12$.
\begin{align*}
  \E\bigg(&\Big(\sum_{i\in E_n} r_i \mu^{(n)}_i e^*_{(\kappa(n),i)}\Big)
            \Big(T\big(\sum_{j\in E_n} r_j\lambda^{(n)}_j e_{(\kappa(n),j)}\big)\Big)\bigg)\\
          &=\E\Big( \sum_{i,j\in E_n} r_i r_j  \mu^{(n)}_i\lambda^{(n)}_j e^*_j(Te_{(\kappa(n),i)})\Big)
            =\sum_{i\in E_n} \mu^{(n)}_i\lambda^{(n)}_ie^*_{(\kappa(n),i)}(Te_{(\kappa(n),i)})\\
          &> \delta(1-\eta).
\end{align*}
The latter inequality follows from the large diagonal of the operator.

\subsubsection*{Postgame}
\label{sec:postygame}

After the game is completed the conditions~\eqref{enu:simul-game:i} to~\eqref{enu:simul-game:iv} of
\Cref{rem:2.4} are satisfied.

\medskip Now, let $m,n\in\mathbb{N}$, $m < n$.  By the winning strategy of Player II
(see~\eqref{enu:simul-game:iii} of \Cref{rem:2.4}) and~\eqref{eq:G_n}, we obtain
\begin{equation}\label{eq:past-est}
  |x_n^*(Tx_m)|
  \leq \|Tx_m\|\cdot\mathrm{dist}(x_{n}^*,G_{n})
  < \|T\|\sqrt{C+\eta}\cdot\eta_{n}.
\end{equation}
We now estimate $|x_n^*(Tx_m)|$ if $m>n$.\medskip
\begin{mdframed}[backgroundcolor=black!5,leftline=false,rightline=false,topline=false,bottomline=false]
  \begin{minipage}[t]{.5\linewidth}
    \begin{mdframed}[backgroundcolor=cyan!5]
      \begin{mdframed}[backgroundcolor=black!20]
        Postgame for \Cref{T:2.5}.
      \end{mdframed}
      \hfill\newline By the winning strategy of Player II
      (see~\Cref{rem:2.4}~\eqref{enu:simul-game:iv}) together with~\eqref{eq:V_n} yields for all
      $n < m$ that
      \begin{align*}
        |x_n^*(Tx_m)|
        &= |T^*x_n^*(x_m)|\\
        &\leq \|T^*x_n^*\|\cdot \dist(x_m, W_m)\\
        &< \|T^*\|\sqrt{C+\eta}\cdot \eta_m.
      \end{align*}

      By~\eqref{eq:eta_n}, \eqref{eq:past-est} and the above estimate, we obtain
      \begin{align*}
        \sum_{\substack{m,n\in\mathbb{N}\\n\neq m}}& |x^*_m(Tx_n)|\\
                                                   &\leq \|T\|\sqrt{C+\eta}\sum_{n=1}^\infty \eta_{n}
                                                     < \eta.
      \end{align*}
      By \Cref{prop:main}~\eqref{prop:a}, there is an infinite set $\Gamma\subset \N$ so that the
      diagonal operator $D_\Gamma\colon Z_\Gamma\to Z_\Gamma$ given by
      \begin{equation*}
        De_{(k,i)} = x_{(k,i)}^*(Tx_{(k,i)})
      \end{equation*}
      is bounded.  By \Cref{L:2.8}, the identity on $Z_\Gamma$ $(K(\delta - \eta)+\xi)$ factors
      through $D_\Gamma$ for any $\xi>0$.  Hence, if $\eta$ is sufficiently small then by the
      additional assertion in \Cref{prop:main}~\eqref{prop:a}, the identity
      $\Big(\frac{ (\lambda K(\delta-\eta)+\xi)(C+\eta)^{2}}{1-5\lambda (K(\delta -
        \eta)+\xi)\eta}\Big)$-factors through $T$.
    \end{mdframed}
  \end{minipage}
  \begin{minipage}[t]{.5\linewidth}
    \begin{mdframed}[backgroundcolor=magenta!5,leftline=false,rightline=false,topline=false,bottomline=false]
      \begin{mdframed}[backgroundcolor=black!20]
        Postgame for \Cref{T:2.7}.
      \end{mdframed}
      \hfill\newline Note that by~\eqref{eq:past-est}, condition~\eqref{prop:iv} of \Cref{prop:main}
      is satisfied if we replace $\eta_n$ by $\eta/2^{n+2}$ .  Secondly, for $n_0<n$ if
      $w_k\in V^{(n)}_k\subset V^{(n_0+1)}_k$, for $k\in\N$ with $\|(w_k)\|_Z\le 1$ and $z=(w_k)$,
      then by \eqref{E:2.5.3}
      \begin{equation*}
        \big|\bigl(T^*_{n_0}x^*_{n_0}\bigr)(z)\big|
        \le \eta_{ n_0}<\eta/2^{n_0+2}.
      \end{equation*}
      Thus, condition~\eqref{prop:v-alt} of \Cref{prop:main} is satisfied, as well.  By
      \Cref{prop:main}~\eqref{prop:b}, the diagonal operator $D:Z\to Z$ given by
      \begin{equation*}
        De_{(k,i)}
        = x_{(k,i)}^*(Tx_{(k,i)})
      \end{equation*}
      is bounded.  By \Cref{L:2.8}, the identity on $Z$ $(K(\delta - \eta)+\xi)$-factors through $D$
      for any $\xi>0$.  Hence, if $\eta$ is sufficiently small, then by the additional assertion of
      \Cref{prop:main}~\eqref{prop:b}, the identity
      $\Big(\frac{ (\lambda K(\delta-\eta)+\xi)(C+\eta)^{2}}{1-2\lambda
        \sqrt{C+\eta}(3+\|T\|)K(\delta - \eta)+\xi)\eta}\Big)$-factors through $T$.
    \end{mdframed}
  \end{minipage}
\end{mdframed}
\medskip Recall that the function $K\colon (0,\infty)\to\mathbb{R}$ is continuous
(see~\cite[Remark~3.11]{lechner:motakis:mueller:schlumprecht:2018}).  As we could have picked $\eta$
and $\xi$ arbitrarily close to zero we deduce that the identity on $Z$ almost
$\lambda K(\delta)C^{2}$-factors through $T$.\qed

\section{Heterogenous $\ell^\infty$-sums of classical Banach spaces}
\label{sec:heter-sums-class}

Given $1 < p_0 < p_1 < \infty$, we define the sets of Banach spaces
\begin{align}
  &\mathcal{W}
    =\{ L^p,\!H^p,\!\vmo,\!\vmo(H^r),\!H^p(H^q),\!L^r(L^s) : 1\kleq p,q\kle\infty, p_0\kleq r,s
    \kleq p_1\}
    \label{eq:W_k-dfn}\\
  & \mathcal{X}
    =\{ L^r, H^p, \vmo, \vmo(H^r), H^p(H^q), L^r(L^s) : p_0 \kleq p,q \kle \infty, p_0 \kleq r,s \kleq p_1\}.
    \label{eq:X_k-dfn}
\end{align}

In this section, we use the following notation: Assume that we are given a sequence of Banach spaces
$(V_k)$ in either $\mathcal{W}$ or $\mathcal{X}$.  Then for each $k,j\in\mathbb{N}$, $e_{k,j}$
denotes the $j$-th Haar function whenever $V_k\in\{L^r, H^p, \vmo\}$, and $e_{k,j}$ denotes the
$j$-th biparameter Haar function if $V_k\in\{\vmo(H^r), H^p(H^q), L^r(L^s)\}$.  For details on the
enumeration of the biparameter Haar system we refer to~\cite{laustsen:lechner:mueller:2015}; see
also~\cite{lechner:motakis:mueller:schlumprecht:2018}.

\begin{cor}\label{cor:curved:1}
  Let $(W_k)_{k=1}^\infty$ denote a sequence of Banach spaces in $\mathcal{W}$ and let
  $T: \ell^\infty(W_k:k\kin\N)\to\ell^\infty(W_k:k\kin\N)$ be bounded and linear, with
  \begin{equation*}
    \delta
    =\inf_{k,j\in\N} \big| e_{k,j}^* (Te_{k,j})\big|
    > 0.
  \end{equation*}
  Then for each sequence of infinite subsets $(\Omega_l)$ of $\mathbb{N}$, there is an infinite
  $\Gamma\subset \N$ so that $\Gamma\cap\Omega_l$ is infinite for all $l\in\mathbb{N}$ and the
  identity on $\ell^\infty(W_k:k\in\Gamma)$ $\frac{C}{\delta}$-factors through $T$, where $C$
  depends only on $p_0,p_1$.
\end{cor}

\begin{proof}
  First, we note that by~\cite[Proof of Theorem~5.1, Theorem~5.2, Theorem~5.3, Theorem~6.1,
  Remark~6.6]{lechner:motakis:mueller:schlumprecht:2018} that the one- or two-parameter Haar system
  in each of the spaces $L^p$, $H^p$, $H^p(H^q)$, $1\leq p,q < \infty$ is $C$-strategically
  reproducible for some universal constant, and $L^r(L^s)$, $p_0\leq r,s\leq p_1$ is
  $C_{p_0,p_1}$-strategically reproducible.  By \Cref{proposition: predual}, the Haar system is
  strategically reproducible in $\vmo$, and the biparameter Haar system is
  $C_{p_0,p_1}$-strategically reproducible in $\vmo(H^p)$.  In each of the spaces in $\mathcal{W}$,
  the Haar system has the $\frac{C}{\delta}$-diagonal factorization property (for the $L^1$ case we
  refer to \cite[Proposition 6.2]{lechner:motakis:mueller:schlumprecht:2018}, for the other cases,
  this follows by unconditionality) for some universal constant $C$.  The assertion follows from
  \Cref{T:2.5}.
\end{proof}

Before we can proceed to our next application, we need the following observation.

\begin{lem}\label{lem:curved}
  Let $(X_k)_{k=1}^\infty$ denote a sequence of Banach spaces in $\mathcal{X}$.  Then the array
  $(e_{k,j})$ is uniformly asymptotically curved.
\end{lem}

\begin{proof}
  For each $k\in\mathbb{N}$, let $(f_{(k,j)})_j$ be a normalized block basis of $(h_I)$.  Let
  $k\in\mathbb{N}$ be fixed for now.  Since the $(f_{(k,j)})_j$ have disjoint Haar spectra, observe
  that for each $n\in\mathbb{N}$
  \begin{equation*}
    \Bigl\| \sum_{j=1}^n f_{(k,j)} \Bigr\|_{X_k}
    \leq C \int_0^1 \Bigl\|\sum_{j=1}^n r_j(t)f_{(k,j)}\Bigr\|_{X_k}dt
    \leq C \Bigl(\sum_{j=1}^n\|f_{(k,j)}\|_{X_k}^r\Bigr)^{1/r},
  \end{equation*}
  where $r_1,\ldots,r_K$ are independent Rademacher functions and $r = p_0$ if $X_n\in\{L^p, H^p\}$
  and $r = 2$ if $X_n=\vmo$~\cite[Proposition 5.1.1, p.~268]{mueller:2005}.  The constant $C$
  depends on $p_0$ and $p_1$ if $X_n = L^p$, and $C=1$ if $X_n\in\{H^p,\vmo\}$.  In either case, we
  obtain
  \begin{equation*}
    \sup_k \Bigl\|\sum_{j=1}^n f_{(k,j)}\Bigr\|_{X_k}
    \leq C n^{1/p_0}.
  \end{equation*}

  Moving on to the biparameter spaces, let $r'$ be such that $\frac{1}{r}+\frac{1}{r'} = 1$.  By
  \Cref{pro:r-estimate}, $H^1(H^{r'})$ satisfies a lower $r'$-estimate with constant $1$; hence by
  \Cref{lem:asymtotically-curved:1} and \Cref{lem:asymtotically-curved:2}, the array formed by the
  biparameter Haar system in the predual of $H^1(H^{r'})$, which is $\vmo(H^r)$, is uniformly
  asymptotically curved.  For the spaces $H^p(H^q)$ and $L^r(L^s)$ we refer to
  \Cref{pro:r-estimate}, \Cref{remark:Lr(Ls)} and \Cref{lem:asymtotically-curved:2}.
\end{proof}

\Cref{lem:curved} and \Cref{T:2.7} combined yield the following factorization result.
\begin{cor}\label{cor:curved:2}
  Let $(X_k)_{k=1}^\infty$ denote a sequence of Banach spaces in $\mathcal{X}$, and let
  $T: \ell^\infty(X_k:k\kin\N)\to\ell^\infty(X_k:k\kin\N)$ be bounded and linear, with
  \begin{equation*}
    \delta
    =\inf_{k,j\in\N} \big| e_{k,j}^* (Te_{k,j})\big|
    > 0.
  \end{equation*}
  Then the identity on $\ell^\infty(X_k:k\kin\N)$ $\frac{C}{\delta}$-factors through $T$, where $C$
  depends only on $p_0,p_1$.
\end{cor}

\begin{proof}
  Let $(e_{k,j})$ be defined as in \Cref{lem:curved}.  First, note that $(e_{k,j})_j$ has
  unconditional basis constant $C_{p_0,p_1}$ for each $k\in\mathbb{N}$; hence, $(e_{k,j})_j$ has the
  $\frac{C_{p_0,q_1}}{\delta}$-diagonal factorization property for all $k\in\mathbb{N}$.  Secondly,
  for the simultaneous strategical reproducibility of the array $(e_{k,j})$, we refer to the
  argument presented in the proof of \Cref{cor:curved:1}.  Finally, applying \Cref{lem:curved} and
  \Cref{T:2.7} concludes the proof.
\end{proof}

\section{Final remarks and open questions}
\label{sec:final-remarks-open}

Our first question asks whether \Cref{T:2.7} can be true if we drop the assumption that the array
$(e_{(k,j)})$ is uniformly asymptotically curved.
\begin{question}\label{question:1}
  Assume that
  \begin{enumerate}[(i)]
  \item the basis constant of $(e_{(k,j)})_j$, is at most $\lambda$ in $X_k$, for each $k\in\N$;
  \item $(e_{(k,j)})_j$ has the $K$-diagonal factorization property in $X_k$, for each $k\in\N$;
  \item the array $(e_{(k,j)})_{k,j}$ is simultaneously $C$-strategically reproducible in $Z$.
  \end{enumerate}
  Is it true that the array $(e_{(k,j)})$ has the factorization property in $Z = \ell^\infty(X_k)$?
\end{question}

The following example exhibits a sequence of spaces the array of which is \emph{not} uniformly
asymptotically curved, yet the array has the factorization property in $Z$.
\begin{example}\label{example:1}
  Let $(p_k)$ be a dense sequence in $[1,\infty)$.  Then
  \begin{enumerate}[(i)]
  \item the basis constant of $(h_{k,I})_I$ is $1$ in $L^{p_k}$;
  \item $(h_{k,I})_I$ has the $\frac{1}{\delta}$-diagonal factorization
    property~\cite[Remark~6.6]{lechner:motakis:mueller:schlumprecht:2018};
  \item the array $(h_{k,I})$ is simultaneously $1$-strategically
    reproducible~\cite[Remark~6.6]{lechner:motakis:mueller:schlumprecht:2018};
  \item the array $(h_{k,I})$ is \emph{not} uniformly asymptotically curved.
  \end{enumerate}
  In contrast, the array $(h_{k,I})$ has in $Z=\ell^\infty(L^{p_k}:k\kin\N)$ the factorization
  property.
\end{example}
In order to show this is indeed true, we need the following proposition.
\begin{prop}\label{prop:dense}
  Let $(p_k)$ and $(q_k)$ denote two dense sequences in $[1,\infty)$.  Then
  $\ell^\infty(L^{p_k}:k\kin\N)$ is isomorphic to $\ell^\infty(L^{q_k}:k\kin\N)$.
\end{prop}

\begin{proof}
  In the first step of this proof we fix $p\in[1,\infty)$ and a dense sequence $(q_k)_k$ in
  $[1, \infty)$ and show that $L^p$ is isomorphic to a $2$-complemented subspace of
  $\ell^\infty(L^{q_k}:k\kin\N)$.  By the density of $(q_k)_k$ we may pick a subsequence
  $(q_{k_n})_n$ so that for each $n\in\mathbb{N}$ the identity map $I_n:L^{q_{k_n}}_n\to L^p_n$ is a
  $2$-isomorphism.  We then define $T:L^p\to \ell^\infty(L^{q_{k_n}}:n\kin\N)$ given as follows: if
  $f\in L^p$ can be written as $f = \sum_j\sum_{|I| = 2^{-j}}a_Ih_I$ then
  $Tf = (\sum_{j\leq n}\sum_{|I| = 2^{-j}}a_Ih_I)_n$.  Then clearly, $T$ is a $2$-isomorphic
  embedding of the image.  We will now show that the image of $T$ is $4$-complemented in
  $\ell^\infty(L^{q_{k_n}}:n\kin\N)$.  To see this, we define
  $Q\colon \ell^\infty(L^{q_{k_n}}:n\in\N)\to L^p$ by
  \begin{equation*}
    Q\Bigl( \sum_{j=0}^\infty\sum_{|I|=2^{-j}} a_{n,I} h_I \Bigr)_n
    = \sum_{j=0}^\infty\sum_{|I|=2^{-j}} (\lim_{n\in \mathcal{U}} a_{n,I}) h_I,
  \end{equation*}
  where $\mathcal{U}$ is some fixed non-principal ultrafilter on $\mathbb{N}$.  For each
  $m\in\mathbb{N}$, define $R_m\colon L^p\to L^p$ by
  \begin{equation*}
    R_m\Bigl( \sum_{j=0}^\infty\sum_{|I|=2^{-j}} a_{I} h_I \Bigr)
    = \sum_{j=0}^m\sum_{|I|=2^{-j}} a_{I} h_I .
  \end{equation*}
  We now verify that $Q$ is indeed well defined.  Observe that for
  $(f_n)_n\in\ell^\infty(L^{q_{k_n}}:n\kin\N)$, where
  $f_n = \sum_{j=0}^\infty a_{n,I} h_I\in L^{q_{k_n}}$, $n\in\N$, we have
  \begin{align*}
    \bigl\| R_m Q ((f_n)_n) \bigr\|_{L^p}
    &= \Bigl\| \sum_{j=0}^m\sum_{|I|=2^{-j}} (\lim_{n\in\mathcal{U}} a_{n,I}) h_I \Bigr\|_{L^p}
      = \lim_{n\in\mathcal{U}} \Bigl\| \sum_{j=0}^m\sum_{|I|=2^{-j}} a_{n,I} h_I \Bigr\|_{L^p}\\
    &= \lim_{n\in\mathcal{U}} \| R_m f_n \|_{L^p}
      \leq \lim_{n\in\mathcal{U}} \| R_n f_n \|_{L^p}\\
    &\leq 2\sup_{n} \| R_n f_n \|_{L^{q_{k_n}}}
      \leq 2 \| (f_n) \|_{\ell^\infty(L^{q_{k_n}}:n\in\N)}.
  \end{align*}
  Evidently, $QTf = f$ for all $f\in L^p$, hence, the image of $T$ is $4$-complemented.  Moreover,
  $T$ can be extended to a $2$-isomorphism $\tilde T:L^p\to\ell^\infty(L^{q_k}:k\kin\N)$ with
  $4$-complemented image.  This type of argument goes back to Johnson~\cite{johnson:1972}.

  In the second step we show that if $(p_k)_k$, $(q_k)_k$ are as in the assumption then
  $\ell^\infty(L^{p_k}:k\kin\N)$ is $2$-isomorphic to a $4$-complemented subspace of
  $\ell^\infty(L^{q_k}:k\kin\N)$.  Indeed, we may decompose $\mathbb{N}$ into infinite disjoint sets
  $(M_n)_n$ so that for all $n\in\mathbb{N}$ the sequence $(p_k)_{k\in M_n}$ is dense in
  $[1,\infty)$.  By the first step, for each $n\in\mathbb{N}$ we can find a $2$-embedding of
  $L^{p_n}$ into $\ell^\infty(L^{p_k}:k\kin M_n)$ with $4$-complemented image.  The second step then
  easily follows.
  
  For the final step we fix a dense sequence $(q_k)_k$ in $[1,\infty)$ with the property that each
  term $q_k$ is repeated infinitely many times.  This implies that the space
  $X = \ell^\infty(L^{q_k}:k\kin\N)$ is isometrically isomorphic to $\ell^\infty(X)$, \ie, it
  satisfies the {\em accordion property} (see~\cite[II.B.24]{wojtaszczyk:1991}).  To conclude, we
  show that for an arbitrary dense sequence $(p_k)_k$ the space $V = \ell^\infty(L^{p_k}:k\kin\N)$
  is isomorphic to $X$.  Indeed, by the second step we have that $X$ is complemented in $V$ and $V$
  is complemented in $X$.  By the accordion property of $X$ we deduce that $X$ is isomorphic to $V$.
\end{proof}

\begin{proof}[Verification of~\Cref{example:1}]
  Given an operator $T\colon Z\to Z$ with large diagonal, choose infinite sets
  $\Omega_k\subset\mathbb{N}$ so that $p_k = \lim_{j\in\Omega_k} p_j$, $k\in\mathbb{N}$.  By
  \Cref{T:2.5} there exists an infinite set $\Gamma\subset\mathbb{N}$ for which $\Gamma\cap\Omega_k$
  is infinite for each $k\in\mathbb{N}$ so that $I_{Z_\Gamma}$ factors through $T$.  Since
  $\{p_\gamma : \gamma\in\Gamma\}$ is again dense in $[1,\infty)$, we obtain by \Cref{prop:dense}
  that $Z_{\Gamma}$ is isomorphic to $Z$.  Hence, $I_Z$ factors through $T$.
\end{proof}

An interesting special case of \Cref{question:1} is the following
\begin{question}\label{question:2}
  Assume that $(p_k)$, $1\leq p_k < \infty $ either converges to $1$ or diverges to $\infty$.  Does
  the array $(h_{k,I})$ have the factorization property in $Z=\ell^\infty(L^{p_k}:k\kin\N)$?
\end{question}

\bibliographystyle{abbrv}
\bibliography{bibliography}
% \nocite{*}

\end{document}